\documentclass[11pt]{amsart}
\usepackage[T1,T2A]{fontenc}
\usepackage[utf8]{inputenc}
\usepackage{graphicx} 
\usepackage{amsmath,amsfonts,amssymb,amsthm,mathtools,pifont,bm,float}
\usepackage{mathrsfs}  
\usepackage{tikz-cd,quiver,enumitem}

\usepackage[margin=1.5in]{geometry}

\usepackage{tabularx,hhline,multirow,caption,pbox,multicol,array,longtable,booktabs} 

\usepackage{verbatim,fancyvrb,fvextra} 
\usepackage{soul} 

\usepackage{url}
\usepackage{hyperref,xcolor}
\hypersetup{
    colorlinks=true,
    linkcolor=black, 
    urlcolor=black,
    citecolor=black,
    }

\usepackage[dvipsnames]{xcolor}

\newtheorem{thm}{Theorem}[section]
\newtheorem{lemma}[thm]{Lemma}
\newtheorem{prop}[thm]{Proposition}
\newtheorem{coro}[thm]{Corollary}
\newtheorem{defn}[thm]{Definition}
\newtheorem{example}[thm]{Example}

\newtheorem{remark}[thm]{Remark}
    
\newcommand{\am}{\operatorname{Am}}
\newcommand{\aut}{\operatorname{Aut}}
\newcommand{\im}{\operatorname{im}}
\newcommand{\pic}{\operatorname{Pic}}
\newcommand{\wg}{\operatorname{WG}}
\newcommand{\pp}{\mathbb{P}}
\newcommand{\zz}{\mathbb{Z}}
\newcommand{\cc}{\mathbb{C}}
\newcommand{\autp}{\operatorname{AutP}}

\newcommand{\neo}{\overline{\operatorname{NE}}}
\newcommand{\nef}{\operatorname{Nef}}

\numberwithin{equation}{section} 

\author{Shreya Sharma}
\address{Shreya Sharma, University of South Carolina, SC, USA.}
\email{shreyas@email.sc.edu}

\title{Actions on the Picard group of smooth Fano threefolds}

\begin{document}

\begin{abstract}
    We classify the possible images of the action of the group of automorphisms of a smooth Fano threefold on its Picard group. We also study the first group cohomology of the Picard group for families of smooth Fano threefolds.
\end{abstract}

\maketitle


\section{Introduction}
Let $X$ be a smooth projective variety over $\mathbb{C}$. We say $X$ is Fano if its anticanonical divisor $-K_X$ is ample. 
Iskovskikh, Mori, and Mukai have classified all deformation families of Fano threefolds, see \cite{isko1, isko2, MM81classtable, mori1983fano, MM86, MMerratum}. There are $105$ such families. A complete list of deformation families with a description of their general smooth member along with other algebraic-geometric information can be found in \cite{fanography}. The classification of Fano threefolds was recently completed in positive characteristic by Asai and Tanaka in \cite{tanaka-I, tanaka-II, asaitanaka-III, tanaka2023IV}. 

The classification of automorphism groups of smooth Fano threefolds is a highly challenging problem, and efforts to understand it are ongoing. The connected components of the identity of automorphism groups of all smooth Fano threefolds were classified in \cite{hilbert, izv19}.
This identity component acts trivially on the Picard group of $X$. Thus, by studying the natural action of automorphism group on the Picard group, one may hope for a deeper understanding of the automorphism group itself. The automorphism groups of smooth cubic threefolds were classified in \cite{wei2019automorphismgroupssmoothcubic}. Several other recent works have investigated the automorphism groups of various families of Fano threefolds; see, for instance, \cite{calabi, cheltsovs2-12family, abe2025automorphismgroupslinearizabilityrational}. 

Suppose that $X$ is a Fano threefold and $G$ is a finite group of automorphisms of $X$. The Leray spectral sequence of the $G$-action on $X$ gives the exact sequence (\cite[\S 3]{kresch2022cohomologyfinitesubgroupsplane})
\begin{align*}
1 \rightarrow \operatorname{Hom}(G,\cc^\times)& \rightarrow \pic(X,G) \rightarrow \pic(X)^G \xrightarrow{\partial} H^2(G,\cc^\times)\to \\
& \operatorname{Br}([X/G])\to H^1(G,\pic(X))\to H^3(G,\cc^\times).
\end{align*}
where $\pic(X,G)$ and $\pic(X)^G$ are the groups of isomorphism classes of $G$-linearized and $G$-invariant line bundles on $X$, respectively. $\operatorname{Br}([X/G])$ is the Brauer group of quotient stack $[X/G]$. The group $\am(X,G):=\im(\partial)$ is called the \emph{Amitsur subgroup}. The three groups, $\am(X,G)$, $H^1(G,\pic(X))$ and $\operatorname{Br}([X/G])$ are important stable equivariant birational invariants (\cite{BCDP, H1isbirinvariant, hassett_tschinkel_odd}). They provide obstructions to equivariant versions of arithmetic analogs of rationality (such as linearizability) and have been a focus of recent interest in \cite{kresch2022cohomologyfinitesubgroupsplane,kresch-unramified-brauergroup,equibrauergroup,tschinkel2025cohomologicalobstructionsequivariantunirationality}. 

The Amitsur subgroup is classified for smooth Fano varieties of dimension 1 and 2 over algebraically closed fields (\cite{BCDP, dolga-modular-curves}). The first cohomology group $H^1(G,\pic(X))$ has also been studied for smooth del Pezzo surfaces, see \cite{manin-cubicforms, urabe-H1computation} for the arithmetic case and \cite{H1isbirinvariant,kresch2022cohomologyfinitesubgroupsplane} for an application to the plane Cremona group. In order to compute both of these groups, it is desirable to determine the induced $G$-action on $\pic(X)$. 

The main objective of this paper is to determine, for each family of smooth Fano threefolds $X$, the possible images $\autp(X)$ of the natural action of $\aut(X)$ on $\pic(X)$. The image of a finite group $G \subseteq \aut(X)$ under this action preserves some combinatorial information associated to the Mori cone of $X$, see \cite[2.8, 2.9]{hmonpaper}. In \cite{Prokhorov-G-FanoII}, Prokhorov lists the families of Fano threefolds for which $\pic(X)^G \cong \zz$, where the $G$-action need not necessarily come from $X$. In a related direction, the authors in \cite[Lemma 2.13]{abban2025kstabilityfano3foldsworld} classify the families of Fano threefolds for which every extremal ray of the Mori cone is $G$-invariant. The actual computation of the image of $G$ for all families of Fano threefolds, however, is still open. In \cite{Mat95}, K. Matsuki defines the Weyl group of a Fano threefold that preserves the same data as the image of $G$ and gives their complete classification (cf. Definition \ref{def:WG}). 

In the table below, we list the families of Fano threefolds with $\rho\leq 5$ that have nontrivial Weyl group, $\wg_X$. For the remaining families, $\wg_X=0$.

\begin{table}[ht]\renewcommand{\arraystretch}{1.2}  
\caption{Weyl groups of Fano threefolds}  \label{weylfano}
\begin{tabularx}{0.80\textwidth}{|X|c|}
\hline
Family & $\wg_X$ \\
\hhline{|=|=|}
\textnumero 2.2, \textnumero 2.6, \textnumero 2.12, \textnumero 2.21, \textnumero 2.32, \textnumero 3.3, \textnumero 3.7, \textnumero 3.9, \textnumero 3.10, \textnumero 3.17, \textnumero 3.19, \textnumero 3.20, \textnumero 3.25, \textnumero 3.31, & $\zz/2\zz$ \\
 \textnumero 4.3, \textnumero 4.4, \textnumero 4.7, \textnumero 4.8, \textnumero 4.10, \textnumero 4.12, \textnumero 4.13, \textnumero 5.2 &  \\
\hline
\textnumero 3.1, \textnumero 3.13, \textnumero 3.27, \textnumero 4.6, \textnumero 5.1 &  $S_3$\\
\hline
\textnumero 4.2  & $(\zz/2\zz)^2$\\
\hline
\textnumero 4.1    &  $S_4$ \\
\hline 
\textnumero 5.3   & $\zz/2\zz \times S_3$ \\
\hline
\end{tabularx}
\end{table}

If $X$ is a Fano threefold with $\rho \geq 6$,  then $X$ is isomorphic to $\pp^1\times S$, where $S$ is a del Pezzo surface of degree at most $5$. In this case, the Weyl group of $X$ coincides with the Weyl group of del Pezzo surface $S$.
The Weyl groups of del Pezzo surfaces of degree at most $5$ are described in \cite[\S\S 6.3–6.7]{dolgaisko_planecremonagroups}.

We observe that the group $\autp(X)$ is contained in Matsuki's Weyl group (cf. Proposition \ref{imagealphainweylgroup}). 
The classification in Table \ref{weylfano} then implies that $\autp(X)=0$ except possibly for the families listed there. For each deformation family with nontrivial $\wg_X$, we either exhibit a smooth member in the family admitting nontrivial automorphisms that induce nontrivial automorphisms of $\pic(X)$, or show that any $G \subseteq \aut(X)$ acts trivially on the Picard group.

\smallskip
The main classification result is as follows. 
\begin{thm}  \label{them:classresult}
  There exists an $X$ in every deformation family of Fano threefolds such that $\autp(X)=\wg_X$ except possibly for families
  \begin{center}
      \textnumero 2.2, \textnumero 6.1, \textnumero 7.1, \textnumero 8.1, \textnumero 9.1, \textnumero 10.1. 
  \end{center}
\end{thm}

In each case, we will see that $\autp(X)$ can in fact be realized by a finite subgroup of the automorphism group for which $\wg_X$ is nontrivial (see \S \ref{proof:classresult}). 

\begin{remark}  \label{autp:2.2exception}\begin{enumerate}[leftmargin=*]
     \normalfont\item[(i)] The result of the theorem is not true for $X$ in family \textnumero 2.2. In this case, the Kleiman-Mori cone $\neo(X)$ is generated by two extremal rays of different types, $C_1$ and $D_1$. Any automorphism of $X$ must keep each invariant, thus acts trivially on $\pic(X)$. So, $\autp(X)=0$, but $\wg_X=\zz/2\zz$.
    \item[(ii)] A smooth Fano threefold in families \textnumero 6.1, \dots , \textnumero 10.1 is isomorphic to $\pp^1 \times S$, where $S$ is a del Pezzo surface of degree at most $5$. For all these cases, Matsuki's Weyl group coincides with the classical Weyl group of root system for the underlying $S$. It is known that the group $\autp(X)$ is usually much smaller than their complete Weyl group (see \S \ref{autP:P1timesdP}). 
 \end{enumerate}
 \end{remark}

As a consequence, we can compute the first cohomology group of Picard group. We have the following result (see \S \ref{sec:H1}).
\begin{thm}  \label{thm:H1main}
    Let $X$ be a smooth Fano threefold with Picard rank $\rho \leq 5$. Let $G\subseteq \aut(X)$ be a finite group. Then $H^1(G,\pic(X))=0$.
\end{thm}
In fact, the group $H^1(G,\pic(X))$ is known even for $\rho \geq 6$ since the first cohomology groups of del Pezzo surfaces are known.

\subsection*{Outline} We begin with some preliminaries in \S \ref{sec:preliminaries}. In \S\ref{imalphasection}, we recall the Weyl group of a smooth Fano threefold as defined by K. Matsuki. We also show that the group $\autp(X)$ is bounded above by this Weyl group. Consequently, we obtain that $\autp(X)$ is trivial for a majority of the families of smooth Fano threefolds. Section \ref{proof:classresult} is divided into eleven subsections, each of which discusses the classification of $\autp(X)$ for a class of families for which $\wg_X$ is nontrivial. In this way, we complete the proof of Theorem \ref{them:classresult}. In \S \ref{sec:H1}, we discuss the first cohomology group of Picard group of smooth Fano threefolds and provide a proof of Theorem \ref{thm:H1main}. Finally, a table summarizing the entire classification is given in \S \ref{table}. 

\subsection*{Acknowledgements} I am grateful to my advisor, Alexander Duncan, for suggesting the problem and for his guidance. 
I also wish to thank Kenji Matsuki for helpful correspondence regarding families \textnumero 3.9, \textnumero~3.13, and \textnumero 4.2. I thank the anonymous referee for helpful comments and suggestions that improved the exposition. 

\section{Preliminaries}  \label{sec:preliminaries}
\subsection{Notation and Conventions}
We work over the field of complex numbers. A deformation family of smooth Fano threefolds is denoted by \textnumero $\rho . N$. Here $\rho$ denotes the Picard rank of the family and $N$ is the number in the classification table in \cite{MM81classtable, MMerratum}. Throughout we work with smooth members of these families. So, by a Fano threefold $X$ in a deformation family, we will mean a smooth member of that family. Additionally, we only work with smooth varieties. 

Consider the natural action of $\operatorname{Aut}(X)$ on $\operatorname{Pic}(X)$ by pullback of line bundles,
\begin{equation}  \label{autpaction}
    \operatorname{Aut}(X) \xrightarrow{\alpha} \operatorname{Aut}\operatorname{Pic}(X).
\end{equation}
The image of a subgroup $G$ of $\aut(X)$, $\alpha(G)$ is denoted as $\autp(X,G)$. For $G=\aut(X)$, we write $\autp(X):=\autp(X,\aut(X))$.

\begin{prop}\label{thm:lefschetz}
   Let $Y$ be a smooth Fano variety of dimension at least $4$. Let $X$ be a smooth ample divisor (or a complete intersection) on $Y$ such that $X$ is a Fano threefold. Then $\pic(Y)\cong \pic(X)$.
\end{prop}
\begin{proof}
From \cite[Proposition 1.15]{isko1}, we have $\pic(Y)=H^2(Y,\zz)$ and $\pic(X)=H^2(X,\zz)$.
    By the Lefschetz Hyperplane Theorem, \cite[3.1]{Laz}, we recall that $H^2(Y,\zz)\cong H^2(X,\zz)$, so the result follows.
\end{proof}

\subsection{Blow-ups} Let $X$ be a Fano threefold obtained by blowing up a Fano threefold $Y$ along a subvariety $Z\subset Y$.
Suppose that $Z = \sqcup_i Z_i$ is the disjoint union of irreducible subvarieties $Z_i$. Let $E_i$ denote the exceptional divisor of $Z_i$ and write $E:= \cup_i E_i$. The notation $\aut(Y;Z)$ will be used for the subgroup of automorphisms of $Y$ that preserve $Z$. Then we have

\begin{lemma}    \label{Lemma:autblowup}
    There is a natural map $\aut(Y;Z) \to \aut(X)$. Conversely, $g \in \aut(X)$ comes from $\aut(Y)$ if and only if $g(E)=E$. 
\end{lemma}
In many of our examples, to show that $\autp(X,G)$ is nontrivial, we construct an automorphism in $\aut(X)$ that is a lift of an automorphism of $Y$ that preserves $Z$ but acts nontrivially on the exceptional locus $E$. However, in some cases, the nontriviality of $\autp(X,G)$ arises from the existence of an automorphism $\sigma \in \aut(X)$ that does not lift from any automorphism in $\aut(Y;Z)$; see, for instance, Lemma \ref{autP:2-21}. 

\smallskip
We will frequently use a computer algebra system to verify the smoothness and irreducibility of an algebraic subvariety in a multiprojective space. 
\begin{remark}  \label{rmk:magmause}
    \normalfont Let $X:= \mathbb{V}(f_1, \dots , f_r) \subset \pp^{N_1}\times \dots \times \pp^{N_r}$, where the $f_i$ are multihomogeneous polynomials whose coefficients lie in $\cc(\{a_{i}\})$ for some indeterminates $a_{i}$. One can use Magma or SageMath to test the smoothness and irreducibilty of $X$ (\cite{magma, sagemath}). Since smoothness is an open condition, the parameters $a_i$ may then be specialized to integers. For an explicit example of such a computation in Magma, see \cite[\S 5.4]{calabi}.
\end{remark}

\section{\texorpdfstring{$\autp(X)$}{AutP(X)} and Weyl group}  \label{imalphasection}
Let $X$ be a Fano threefold. Recall that $\autp(X,G)$ denotes the image of $G \hookrightarrow \aut(X) \xrightarrow{\alpha} \aut(\pic(X))$. In \cite[Definition 2.8]{hmonpaper}, the authors describe a number of combinatorial properties that are preserved by the group $\autp(X,G)$. The Weyl group of a Fano threefold preserves the same type of information, but has the advantage of being completely classified by Matsuki in \cite[III]{Mat95} and \cite{Mat23}. 
 
We now introduce the necessary machinery to define it. Recall the following constructions from \cite[III]{Mat95}. Let 
\[
T:= \operatorname{Spec}\bigoplus_{m\geq 0}H^0(X,\operatorname{Sym}^m(-K_X))
\]
be a cone over $X$.
The total space of $K_X$ is defined as
\[
Y := \underline{\operatorname{Spec}}\bigoplus_{m\geq 0}\operatorname{Sym}^m(-K_X),
\]
where $\underline{\operatorname{Spec}}$ denotes the relative spectrum of $\bigoplus \operatorname{Sym}^m(-K_X)$. This comes with a projection map $f: Y \rightarrow T$ that contracts the zero section of $Y$ to a point in $T$. Note that the zero section of $Y$ can be identified with $X$, so there is an embedding $X \subset Y$ such that the exceptional locus is Exc$(f) = X$. 

We also recall the following definitions from \cite{kawamata}:
\begin{itemize}
    \item $\operatorname{Amp}(Y/T)$ is the convex cone generated by classes of relatively ample divisors in $N^1(Y/T)$. 
    \item The nef cone, $\nef(Y/T)\subset N^1(Y/T)$ is the closure of $\operatorname{Amp}(Y/T)$. It is the dual of Kleiman-Mori cone, $\neo(Y/T)$ with respect to the intersection pairing.
    \item We say $D \in \pic(Y)$ is $f$-movable if 
   \begin{enumerate}
       \item[(i)] $f_*D \neq 0$,
       \item[(ii)] the cokernel of the natural homomorphism $f^*f_*D \to D$ has support of codimension at least $2$. 
   \end{enumerate}
   \smallskip
   We define $\overline{Mov}(Y/T)$ as the closed convex cone in $N^1(Y/T)$ generated by the classes of $f$-movable divisors.
\end{itemize}

Since any two minimal models in a birational class are isomorphic in codimension one, $N^1(Y_i/T)$ are (canonically) isomorphic for any minimal models $Y_i/T$. Let us write $N^1$ for $N^1(Y_i/T)$ for any $i$. The movable cones $\overline{Mov}(Y_i/T)$ can also be identified to give the cone $\overline{Mov}$ in $N^1$ (see \cite[II]{Mat95}). 

The Kawamata-Koll\'ar-Mori-Reid (KKMR) decomposition theorem says that  $\overline{Mov}$ decomposes into finitely many polyhedral cones that are determined by minimal models of $Y/T$. 

\begin{thm}\cite[Theorem II-4, KKMR decomposition]{Mat95}  \label{Mat:wg}
    We have the decomposition 
    \[
    \overline{Mov} = \text{ the closure of } \Big(\coprod_i \nef(Y_i/T) \Big)
    \]
    where the symbol $\coprod$ means that no two cones $\nef(Y_i/T)$ and $\nef(Y_j/T)$ intersect in their interiors for $i\neq j$ and $\# \{i\} < \infty$. 
\end{thm}  

\begin{defn}\cite[III.1]{Mat95}  \label{def:WG}
    \normalfont Let $X$ be a Fano threefold. The Weyl Group of $X$, denoted as $\wg_X$, is characterized as the group of automorphisms of $\pic(Y/T)$
    \begin{enumerate}
        \item[(a)] fixing the class of the exceptional divisor for $f$,
        \item[(b)] preserving the KKMR decomposition. 
    \end{enumerate}
\end{defn}
From \cite[III-3-1]{Mat95}, we see that $N^1(X) \cong N^1(Y/T)$, and thus we may identify $\pic(Y/T) \cong \pic(X)$. This allows us to see $\wg_X$ as a subgroup of $\operatorname{Aut}(\operatorname{Pic}(X))$. 
\vspace{0.3cm}

\begin{prop}\label{imagealphainweylgroup}
   Let $X$ smooth Fano threefold and $G\subseteq \aut(X)$ be a group. Then $\autp(X,G) \subseteq \wg_X$. In particular, we have $\autp(X)\subseteq \wg_X$. 
\end{prop}
\begin{proof}
Let $g\in \aut(X)$ be such that $g^*\in \autp(X)$. Since $Y$ is the total space of $K_X$, $g$ lifts and we get an isomorphism $\phi_g : g^*Y \xrightarrow{\sim} Y$, where $g^*Y := X \times_{X,g} Y$. 
Observe that $g^*Y$ is the total space of $g^*K_X$. Now as $g$ fixes the class of $K_X$, we get an isomorphism $Y \to g^*Y$ for a choice of $\eta :g^*K_X \xrightarrow{\sim} K_X$. Combining, we get an automorphism $\tilde{g}: Y \to Y$.

\[\begin{tikzcd}
	T && T \\
	Y & {g^*Y} & Y \\
	& X & X
	\arrow["\tilde{\phi_g}", from=1-1, to=1-3]
	\arrow["f", from=2-1, to=1-1]
	\arrow["\sim", from=2-1, to=2-2]
    \arrow["\pi", from=2-1, to=3-2]
	\arrow["{\phi_g}", from=2-2, to=2-3]
	\arrow[from=2-2, to=3-2]
	\arrow["f"', from=2-3, to=1-3]
	\arrow["\pi", from=2-3, to=3-3]
	\arrow["g"', from=3-2, to=3-3]
\end{tikzcd}\]
Moreover, the same choice of $\eta$ induces an automorphism 
\begin{align*}
\phi: H^0(X,\operatorname{Sym}^m(-K_X)) &\to H^0(X,\operatorname{Sym}^m(-g^*K_X))\cong H^0(X,\operatorname{Sym}^m(-K_X)) \\
s &\mapsto g^*s \mapsto \eta_*(g^*s).
\end{align*}
This induces $\tilde{\phi_g}: T \xrightarrow{\sim} T$. From the construction, the class of $X$ in $Y$, which is also the exceptional divisor for $f$ is preserved under the action of $g$. 

Consider a relative minimal model $Y_i/T$ of $Y/T$ and its pullback $Y_i\times_T T$ via $\tilde{\phi}_g$. Since $\tilde{\phi}_g$ is an isomorphism, $K_{(Y_i\times_T T)/T}$ is also nef, consequently, $(Y_i \times_T T)/T$ is a (relative) minimal model of $Y/T$ in the same birational class as $Y_i/T$. We can now write $Y_i\times_T T$ as $Y_{g(i)}$ for some index $g(i)$ which may be different from $i$. 

\[\begin{tikzcd}
	{Y_{g(i)}} & {Y_i} \\
	T & T
	\arrow[dashed, "\psi_g", from=1-1, to=1-2]
	\arrow[from=1-1, to=2-1]
	\arrow[from=1-2, to=2-2]
	\arrow["{\tilde{\phi_g}}"', from=2-1, to=2-2]
\end{tikzcd}\]
Let $D \in \operatorname{Amp}(Y_i/T)$. Then $\psi_g^*D$ is also relatively ample on $Y_{g(i)}$. So, we have shown that $\psi_g^*\operatorname{Amp}(Y_i/T)\subseteq \operatorname{Amp}(Y_{g(i)}/T)$. 
In other words, $g^*$ induces a map on the cone $\coprod \nef(Y_i/T)$ that sends an element in the interior of one chamber to the interior of another chamber, and thus the KKMR decomposition is preserved.  
\end{proof}

We observe that $\alpha(G)=\autp(X,G)$ is trivial for all smooth Fano threefolds with rank $\rho =1$. This is because $\operatorname{Pic}(X)$ has a single generator $[K_X]$ which must be fixed under the linear action. Alternatively, since the Weyl group is trivial, $\autp(X,G)$ must also be trivial by Proposition \ref{imagealphainweylgroup}.

\smallskip
From Proposition \ref{imagealphainweylgroup} and Table \ref{weylfano}, we immediately get 
\begin{coro}   \label{im=0forsure}
    Let $X$ be a Fano threefold. Then $\autp(X)=0$ for all deformation families of Fano threefolds except possibly,
    \begin{enumerate}
        \item[(i)] $5$ families with $\rho =2$, 
        \begin{center}
            \textnumero 2.2,\; \textnumero 2.6,\; \textnumero 2.12,\; \textnumero 2.21, \; \textnumero 2.32.
        \end{center}
        \item[(ii)] $12$ families with $\rho =3$,
        \begin{center}
            \textnumero 3.1,\; \textnumero 3.3,\; \textnumero 3.7,\; \textnumero 3.9,\; \textnumero 3.10,\;\textnumero 3.13,\; \textnumero 3.17,\; \textnumero 3.19,
            \; \textnumero 3.20,\; \textnumero 3.25,\; \textnumero 3.27,\; \textnumero 3.31.
        \end{center}
        \item[(iii)] $10$ families with $\rho =4$,
        \begin{center}
            \textnumero 4.1,\; \textnumero 4.2, \; \textnumero 4.3,\; \textnumero 4.4,\; \textnumero 4.6,\; \textnumero 4.7, \textnumero 4.8,\; \textnumero 4.10,\; \textnumero 4.12,\; \textnumero 4.13.
        \end{center}
        \item[(iv)] $8$ families with $\rho \geq 5$,
        \begin{center}
        \textnumero 5.1,\; \textnumero 5.2,\; \textnumero 5.3,\; \textnumero 6.1,\; \textnumero 7.1,\; \textnumero 8.1,\; \textnumero 9.1,\; \textnumero 10.1.
        \end{center}
    \end{enumerate}
\end{coro}

For each of the families listed in the above corollary, we now find whether $\autp(X)=0$ or the upper bound on it given by the Weyl group can be realized for a group $G\subseteq \aut(X)$. 

\section{Classification of \texorpdfstring{$\autp(X,G)$}{AutP(X,G)}}  \label{proof:classresult}
\subsection {Divisors on a product of projective spaces}   \label{im(alpha)divisortype}
Let $X$ be a smooth Fano threefold that is a member of any of the five deformation families
\begin{center}
\textnumero 2.6\;(a),\; \textnumero 2.32,\; \textnumero 3.3,\; \textnumero 3.17, \; \textnumero 4.1.
\end{center}
For each family, we use their Mori-Mukai description in \cite[Tables 2-4]{MM81classtable}. A smooth Fano threefold in \textnumero 2.6 is either isomorphic to a divisor of bidegree $(2,2)$ in $\pp^2\times\pp^2$ (\textnumero 2.6 (a)), or a double cover (\textnumero 2.6 (b), see Lemma \ref{autP:2-6b)}).

If $X$ is a member of \textnumero 2.6 (a) or \textnumero 2.32, it can be described as a divisor on $\pp^2\times\pp^2$. Let $H_1$ and $H_2$ denote the pullbacks to $\pp^2 \times \pp^2$ of the hyperplane classes from the first and second factors of $\pp^2$, respectively.
 Then from Theorem \ref{thm:lefschetz}, we have $\pic(X)\cong\zz[H_1]\oplus \zz[H_2]$. Let $([x_0,x_1,x_2],[y_0,y_1,y_2])$ be coordinates on $\pp^2\times \pp^2$.

\begin{lemma} \label{autP:2-6a)}
There exists an $X$ in \textnumero 2.6 (a) such that $\autp(X)=\wg_X \cong \zz/2\zz$. 
\end{lemma}
\begin{proof}
A member $X$ of the family \textnumero 2.6 (a) is a divisor of bidegree $(2,2)$ on $\pp^2\times\pp^2$. Consider the involution of $\mathbb{P}^2 \times \mathbb{P}^2$ 
\[
\sigma: ([x_0:x_1:x_2],[y_0:y_1:y_2]) \mapsto ([y_0:y_1:y_2],[x_0:x_1:x_2]),
\]
and the equation
\begin{equation}    \label{eqn:2.6a}
    x_0^2 Q_1 + x_1^2 Q_2+ x_2^2 Q_3 +  x_0x_1Q_4 + x_1x_2 Q_5 +x_0x_2Q_6 =0.
\end{equation}
where each $Q_i(y_0,y_1,y_2)$ is homogeneous defined as
\[
a_{i1} y_0^2 + a_{i2} y_1^2 + a_{i3}y_2^2 + a_{i4}y_0y_1 + a_{i5}y_1y_2 + a_{i6}y_0y_2, \hspace{0.5cm} \text{ for } 1\leq i \leq 6.
\]
Assume that $a_{ij} = a_{ji}$ for $1\leq i,j\leq 6$. Then \eqref{eqn:2.6a} is $\sigma$-invariant. For a general choice
of the coefficients $a_{ij}$ satisfying these symmetry conditions,
it defines a smooth divisor $X$ of bidegree $(2,2)$ in $\pp^2\times\pp^2$; for instance, one may choose 
\[
a_{13}=a_{31}=a_{14}=a_{41}=a_{25}=a_{52}=a_{36}=a_{63}=a_{55}=a_{56}=a_{65}=1, 
\]
with all other $a_{ij}=0$, and verify using Remark \ref{rmk:magmause}.
Observe that $\sigma$ acts on $\pic(X)$ by swapping its two generators. Thus, $\autp(X,\langle \sigma \rangle) \cong \zz/2\zz$.  
\end{proof}

\begin{lemma} \label{autP:2-32}
There exists an $X$ in \textnumero 2.32 such that $\autp(X)=\wg_X \cong \zz/2\zz$. 
\end{lemma}
\begin{proof}
A member $X$ of the family \textnumero $2.32$ is a divisor of bidegree $(1,1)$ on $\mathbb{P}^2 \times \mathbb{P}^2$. Consider the involution of $\pp^2\times\pp^2$
\begin{equation}   \label{eqn:involutionforW}
\sigma: ([x_0:x_1:x_2],[y_0,y_1,y_2]) \mapsto ([y_0,y_1,y_2],[x_0,x_1,x_2]).
\end{equation}
Assume $X$ is given by
\begin{equation}  \label{eqn:W}
 x_0y_0 +x_1y_1 + x_2y_2 =0.
\end{equation}
Then $X$ is both smooth and irreducible, which can be verified using Remark \ref{rmk:magmause}. Finally, note that $X$ is also $\sigma$-invariant, and $\sigma$ acts on $\pic(X)$ by swapping its two generators. This gives, $\autp(X,\langle \sigma \rangle)\cong \zz/2\zz$. 
\end{proof}

If $X$ is in family \textnumero 3.3 or \textnumero 3.17, it is a divisor on $\pp^1\times\pp^1\times\pp^2$. From Theorem \ref{thm:lefschetz}, we have $\pic(X)\cong \zz[H_1]\oplus\zz[H_2]\oplus\zz[H_3]$. Here $H_1$, $H_2$, and $H_3$ denote the pullbacks to $\pp^1 \times \pp^1 \times \pp^2$ 
of the hyperplane classes from the first and second factors of $\pp^1\times \pp^1$ and $\pp^2$, respectively. Let $([x_0:x_1],[y_0:y_1],[z_0:z_1,z_2])$ be coordinates on $\mathbb{P}^1\times \mathbb{P}^1\times \mathbb{P}^2$.

\begin{lemma}  \label{autP:3-3}
There exists an $X$ in \textnumero 3.3 such that $\autp(X)=\wg_X \cong \zz/2\zz$. 
\end{lemma}
\begin{proof}
If $X$ is a member of family \textnumero 3.3, it is a smooth divisor of tridegree $(1,1,2)$ on $\mathbb{P}^1\times \mathbb{P}^1\times \mathbb{P}^2$. Consider the involution of $\pp^1\times\pp^1\times\pp^2$
\[
\sigma: ([x_0:x_1],[y_0:y_1],[z_0:z_1:z_2]) \mapsto ([y_0:y_1],[x_0:x_1],[z_0:z_1:z_2]).
\]
Assume that $X$ is given by the equation
\[
    x_0y_0(z_0z_1+z_2^2) + x_1y_1(z_1z_2+z_0^2) + (x_0y_1 + x_1y_0)z_1^2 =0,
\]
then it is $\sigma$-invariant. Using Remark \ref{rmk:magmause}, we find that $X$ is both smooth and irreducible. Finally, note that $\sigma$ induces an action on $\pic(X)$ that swaps $H_1$ and $H_2$. Thus, $\autp(X,\langle \sigma \rangle)\cong \zz/2\zz.$
\end{proof}

\begin{lemma}  \label{autP:3-17}
There exists an $X$ in \textnumero 3.17 such that $\autp(X)=\wg_X \cong \zz/2\zz$. 
\end{lemma}
\begin{proof}
If $X$ is a member of the family \textnumero $3.17$, it can be described as a divisor of tridegree $(1,1,1)$ on $\mathbb{P}^1 \times \mathbb{P}^1 \times \mathbb{P}^2$. Consider the involution
\[
\sigma: ([x_0:x_1],[y_0:y_1],[z_0:z_1:z_2]) \mapsto ([y_0:y_1],[x_0:x_1],[z_0:z_1:z_2])
\]
on $\pp^1\times\pp^1\times\pp^2$. 
Now suppose that $X$ is given by the equation
\[
x_1y_1z_0 - x_1y_0z_1 - x_0y_1z_1 + x_0y_0z_2=0.
\]
Then $X$ is $\sigma$-invariant and it can be verified to be smooth and irreducible (see Remark \ref{rmk:magmause}). Observe that $\sigma$ induces an action on $\pic(X)$ that swaps its two generators $H_1$ and $H_2$. Thus, $\autp(X,\langle \sigma\rangle)=\zz/2\zz$. 
\end{proof}

A member of family \textnumero $4.1$ is described as a divisor on  $\pp^1\times\pp^1\times\pp^1\times\pp^1$ of multidegree $(1,1,1,1)$. From Theorem \ref{thm:lefschetz}, $\pic(X) \cong \zz[H_1]\oplus \zz[H_2]\oplus \zz[H_3]\oplus \zz[H_4]$, where $H_i$ is the pullback to $(\pp^1)^4$ of a hyperplane on $i$th copy of $\pp^1$ for $1 \leq i \leq 4$. 
\begin{lemma}  \label{autP:4-1}
There exists an $X$ in \textnumero 4.1 such that $\autp(X)=\wg_X \cong S_4$. 
\end{lemma}
\begin{proof} 
 Consider automorphisms of $(\pp^1)^4$
\begin{align*}
\sigma: ([x_0:x_1], [y_0:y_1], [z_0:z_1],[w_0:w_1]) \mapsto ([y_0:y_1], [z_0:z_1], [w_0:w_1],[x_0:x_1])\\
\tau: ([x_0:x_1], [y_0:y_1], [z_0:z_1],[w_0:w_1]) \mapsto ([y_0:y_1], [x_0:x_1], [z_0:z_1],[w_0:w_1]),
\end{align*}
where $([x_0:x_1], [y_0:y_1], [z_0:z_1],[w_0,w_1])$ are projective coordinates on $(\pp^1)^4$.
Assume that $X$ is given by the equation
\begin{multline*}
 x_0y_0z_0w_0 +  (x_0y_0z_0w_1 + x_1y_0z_0w_0 + x_0y_1z_0w_0 + x_0y_0z_1w_0)+ \\
 -(x_0y_0z_1w_1 + x_1y_0z_0w_1 + x_1y_1z_0w_0 + x_0y_1z_1w_0+ x_0y_1z_0w_1 + x_1y_0z_1w_0) \\+
 (x_0y_1z_1w_1 + x_1y_0z_1w_1  + x_1y_1z_0w_1 + x_1y_1z_1w_0)+ x_1y_1z_1w_1  =0,
\end{multline*}
Using Remark \ref{rmk:magmause}, we can verify that $X$ is both smooth and irreducible. Moreover, observe that both $\sigma, \tau$ preserve $X$, so we can assume that $\sigma, \tau \in \aut(X)$. Finally, note that $\langle \sigma,\tau\rangle \cong S_4$ acts on $\pic(X)$ by permutations. Thus, $\autp(X,\langle \sigma,\tau\rangle )=S_4$. 

The same equation gives an $X$ for which $\autp(X,G)$ is isomorphic to other subgroups of $S_4$ by choosing $G$ to be an appropriate subgroup of $\langle \sigma,\tau\rangle$. 
\end{proof}

\subsection{Double covers}  \label{im(alpha)fordoublecovers}
We discuss the following families of Fano threefolds whose smooth member is a double cover
\begin{center}
\textnumero 2.6 (b),\; \textnumero 3.1.
\end{center}
For both families, we use their Mori-Mukai description in \cite[Tables 2, 3]{MM81classtable}. The case for family \textnumero 2.6 has already been settled in Lemma \ref{autP:2-6a)}. Nevertheless, we record that the same holds for case~2.6(b), as it can be seen just as easily.

\begin{lemma} \label{autP:2-6b)}
There exists an $X$ in \textnumero 2.6 (b) such that $\autp(X)=\wg_X \cong \zz/2\zz$. 
\end{lemma}
\begin{proof} 
In Mori-Mukai notation, let us write $W$ for a smooth Fano threefold in \textnumero 2.32. Then if $X$ is in \textnumero 2.6 (b), it can be described as a double cover of $W$ with branch locus $B$ given by a smooth member in $|-K_W|$. Let $\pi : X \to W$ be the double covering. 

Recall the involution $\sigma$ from \eqref{eqn:involutionforW}. 
Suppose that $B$ is given by the intersection of hypersurfaces
\[
B:= \begin{cases}
    \sum_{i,j,k,l=0}^2 a_{ijkl}\; x_i x_j y_k y_l =0 \\
    x_0y_0 + x_1y_1 + x_2y_2 =0,
\end{cases} 
\] 
where the first equation defines a divisor $-K_{\pp^2\times\pp^2} = (2,2)$ and the second equation defines $W$ (as in \eqref{eqn:W}), introduced in Lemma~\ref{autP:2-32} as a smooth member of family \textnumero 2.32.
If we further assume that coefficients satisfy $a_{ijkl} = a_{klij}$, then $B$ is $\sigma$-invariant. For a general choice of the coefficients satisfying these symmetry conditions, $B$ is smooth. For instance, for the choice of coefficients 
\[
a_{0000}=a_{0122}=a_{0211}=a_{1102}=a_{1212}=a_{2201}=1,
\]
and $a_{ijkl}=0$ for remaining, we verify $B$ to be smooth and irreducible using Remark \ref{rmk:magmause}.
Thus, $\sigma \in \aut(W;B)$. Let $\tilde{\sigma}$ be the preimage of $\sigma$ under the map $\aut(X) \twoheadrightarrow
\aut(W;B)$. Since $\sigma$ swaps the two generators of $\pic(W)$ and $\pic(X) \cong \zz[\pi^*H_1]\oplus \zz[\pi^*H_2]$ (\cite[Theorem 3.8]{MM86}), therefore, $\tilde{\sigma}$ acts on $\pic(X)$ by swapping its two generators.
Thus, we have $\autp(X,\langle \tilde{\sigma}\rangle) = \mathbb{Z}/2\mathbb{Z} $.   
\end{proof} 

\begin{lemma}  \label{autP:3-1}
There exists an $X$ in \textnumero 3.1 such that $\autp(X)=\wg_X \cong S_3$. 
\end{lemma}
\begin{proof}
A smooth Fano threefold $X$ in \textnumero 3.1 is a double cover of $\mathbb{P}^1 \times \mathbb{P}^1 \times \mathbb{P}^1$ with branch locus $B$ a divisor of tridegree $(2,2,2)$. Let $\pi : X \to (\pp^1)^3$ denote the double covering. 

Consider $\sigma, \tau \in \aut((\pp^1)^3)$ defined as
\begin{align*}
\sigma &: ([x_0:x_1],[y_0,y_1],[z_0:z_1]) \mapsto ([y_0:y_1],[x_0:x_1],[z_0:z_1])\\
\tau &: ([x_0:x_1],[y_0,y_1],[z_0:z_1]) \mapsto ([y_0:y_1],[z_0:z_1],[x_0:x_1]).
\end{align*}
Let $\mathcal{G}:=\langle \sigma,\tau  \rangle$. Consider the equation 
\[
x_0^2y_0^2z_0z_1 + x_0^2y_0y_1z_0^2 + x_0x_1(y_0^2z_0^2 + y_1^2z_1^2) + x_1^2y_0y_1z_1^2 + x_1^2y_1^2z_0z_1=0.
\]
It is $\mathcal{G}$-invariant, and we verify it to be both smooth and irreducible (Remark \ref{rmk:magmause}). Thus, the equation defines $B$. Let $\widetilde{\mathcal{G}}$ be the preimage of $\mathcal{G}$ under the map $\aut(X) \twoheadrightarrow \aut((\pp^1)^3;B)$. From \cite[Theorem 3.8]{MM86}, we have $\pic(X)\cong \zz[\pi^*H_1]\oplus \zz[\pi^*H_2]\oplus \zz[\pi^*H_3]$. Observe that $\mathcal{G}\cong S_3$ acts on $\pic(X)\cong \zz^3$ by permutations, so $\tilde{\mathcal{G}}$ acts on $\pic(X)$ via permutations as well and thus,
we get $\autp(X,\widetilde{\mathcal{G}}) = S_3$.  
\end{proof}

\subsection{Blow-ups of \texorpdfstring{$\mathbb{P}^3$}{P3}} \label{im(alpha)forblowupofP3}
We consider families \textnumero $2.12$ and \textnumero $3.25$. A smooth member $X$ of these two families can be described as a blow-up of $\mathbb{P}^3$ in a curve. 

\begin{lemma}  \label{autP:2-12}
There exists an $X$ in \textnumero 2.12 such that $\autp(X)=\wg_X \cong \zz/2\zz$. 
\end{lemma}
\begin{proof}
Mori-Mukai describe a member $X$ of family \textnumero 2.12 as the blow up of $\pp^3$ in a curve of degree $6$ and genus $3$ which is an intersection of cubics. We use an alternate description of $X$. In \cite[\S 5.4]{calabi}, the authors give an alternative description of $X$ as a complete intersection of three divisors of bidegree $(1,1)$ in $\pp^3\times\pp^3$. We consider the particular member given there by
\[
x_0y_1+x_1y_0-\sqrt{2}x_2y_2 \;=\; x_0y_2+x_2y_0-\sqrt{2}x_3y_3 \;=\; x_0y_3+x_3y_0-\sqrt{2}x_1y_1 \; =\;0
\]
where $[x_0:x_1:x_2:x_3]$ and $[y_0:y_1:y_2:y_3]$ are coordinates on two factors of $\mathbb{P}^3\times \mathbb{P}^3$. This intersection is invariant under the involution 
\[
\sigma : ([x_0:x_1:x_2:x_3],[y_0:y_1:y_2:y_3]) \mapsto ([y_0:y_1:y_2:y_3],[x_0:x_1:x_2:x_3])
\]
of $\mathbb{P}^3\times \mathbb{P}^3$, so we can assume $\sigma \in \aut(X)$.
Observe that $\sigma$ swaps the two generators of $\pic(X)\cong \zz^2$ (Theorem \ref{thm:lefschetz}). Thus, we have $\autp(X,\langle \sigma \rangle)=\zz/2\zz$.  
\end{proof} 

\begin{lemma}  \label{autP:3-25}
There exists an $X$ in \textnumero 3.25 such that $\autp(X)=\wg_X \cong \zz/2\zz$. 
\end{lemma}
\begin{proof} 
From the Mori-Mukai description, if $X$ is a member of family \textnumero 3.25, it can be described as the blow-up of $\mathbb{P}^3_{x_0, x_1, x_2, x_3}$ in the disjoint union of two lines, say $\ell$ and $m$ (\cite[Table 3]{MM81classtable}). Let
\begin{align*}
    \ell := \{ x_0 = x_1 = x_2 =0 \},  \hspace{1.0cm} m := \{ x_1 = x_2 = x_3 =0\}.
\end{align*}
be two lines in $\pp^3$ and $ \sigma :[x_0:x_1:x_2:x_3] \mapsto [x_3:x_2:x_1:x_0]$ be an involution. Then $\sigma$ keeps $\ell \cup m$ invariant but interchanges the two lines on $\pp^3$. Note that $\pic(X)=\zz[H]\oplus \zz[E_\ell]\oplus\zz[E_m]$, where $H$ is the pullback of a general hyperplane in $\pp^3$ and $E_\ell, E_m$ stand for exceptional divisors of the blow-up corresponding to lines $\ell$ and $m$, respectively. Therefore, by Lemma \ref{Lemma:autblowup}, $\sigma$ lifts to an involution $\tilde{\sigma}$ of $X$ that swaps $E_{\ell}$ and $E_m$. So, $\autp(X,\langle \tilde{\sigma}\rangle) = \mathbb{Z}/2\mathbb{Z}$. 
\end{proof} 
We remark that the Fano threefold $X$ in \textnumero 3.25 is toric and as such, $\autp(X)$ can also be found using the description of its fan. See $\S \ref{subsec:toric}$ for examples that illustrate this.

\subsection{Blow-ups of \texorpdfstring{$Q$}{Q}}   \label{im(alpha)forblQ}
Let $Q \subset \mathbb{P}^4$ be a smooth quadric threefold with coordinates $[x_0:x_1:x_2:x_3:x_4]$.
 The following four families 
\begin{center}
\textnumero 2.21, \; \textnumero 3.10, \; \textnumero 3.19, \; \textnumero 3.20
\end{center}
have a smooth member $X$ that is a blow-up of $Q$ in a curve $C$. 

\begin{lemma}  \label{autP:2-21}
There exists an $X$ in \textnumero 2.21 such that $\autp(X)=\wg_X \cong \zz/2\zz$. 
\end{lemma}
\begin{proof}
If $X$ is in \textnumero 2.21, the center of the blow-up is a twisted quartic $C$ (\cite[Table 2]{MM81classtable}). We recall the following from \cite[\S 5.9]{calabi}-- assume $Q$ is given by 
\[
x_1x_3 - s^2x_0x_4 + (s^2-1)x_2^2=0,
\]
where $s \in \mathbb{C} \smallsetminus \{0, \pm 1\}$. It contains the twisted curve $C$ parametrized by $\pp^1_{u,v} \hookrightarrow \pp^4$
\[
[u:v] \mapsto [u^4:u^3v:u^2v^2:uv^3:v^4].
\]
Then the authors show the existence of an involution $\sigma \not\in \aut(Q;C)$ such that it is a lift of a birational involution $\tau: Q\to Q$ given by (\cite[Remark 5.52]{calabi})
\begin{align*}
    \tau: [x_0:x_1:x_2:x_3:x_4]  \mapsto [x_0x_2-x_1^2 : s(x_0x_3-x_1x_2): & s^2(x_0x_4-x_2^2): \\
    & s(x_1x_4-x_2x_3): x_2x_4-x_3^2].
\end{align*}
and $\aut(X)$ is generated by $\aut(Q;C)$ and $\sigma$. Suppose that $\pi : X \to Q$ is the blow-up map. Then $\pic(X)=\zz[\pi^*H]\oplus \zz[E]$, where $H$ is the pullback to $Q$ of a hyperplane on $\pp^4$ and $E$ is the exceptional divisor. 
The involution $\sigma$ acts on $\pic(X)$ via 
\[
\pi^*H \mapsto 2\pi^*H-E, \; \; E\mapsto 3\pi^*H -2E.
\]
This shows that $\autp(X,\langle \sigma\rangle)=\zz/2\zz$.
\end{proof}
If $X$ is in one of the families, \textnumero 3.10, \textnumero 3.19, or \textnumero 3.20, it is the blow-up of $Q$ in the disjoint union of two conics, points, or lines, respectively (\cite[Table 3]{MM81classtable}). Let $\pi :  X \to Q$ be the blow-up map. In each case, we have  $\pic(X)=\zz[\pi^*H]\oplus\zz[E_1]\oplus\zz[E_2]$, where $H$ is the pullback to $Q$ of a general hyperplane in $\pp^4$ and $E_1$, $E_2$ are the exceptional divisors corresponding to two conics (resp. points, lines). 

\begin{lemma}  \label{autP:3-10}
There exists an $X$ in \textnumero 3.10 such that $\autp(X)=\wg_X \cong \zz/2\zz$. 
\end{lemma}
\begin{proof} If $X$ is in \textnumero 3.10, it is the blow-up of $Q$ in the disjoint union of two conics, $C_1$ and $C_2$. Let $Q$ be a smooth quadric given by $x_4^2 + x_0x_1 + x_2x_3=0$ in $\pp^4$. Consider an involution 
\[\sigma: [x_0:x_1:x_2:x_3:x_4] \mapsto [x_2:x_3:x_0:x_1:x_4]\]
in $\aut(\pp^4;Q)$ and curves
\begin{align*}
    C_1 := \{ x_0 = x_1 = Q=0 \} , \; \; 
    C_2 := \{x_2 = x_3 = Q=0\}.
\end{align*}
Then $C_1$ and $C_2$ are clearly disjoint and smooth. Note that $\sigma$ acts on $Q$ such that it swaps the two curves and keeps their disjoint union invariant. By Lemma \ref{Lemma:autblowup}, it now follows that $\sigma$ lifts to an automorphism $\tilde{\sigma} \in \aut(X)$ such that it swaps the exceptional divisors, $E_1$ and $E_2$ of the blow-up. Thus, we have shown that $\autp(X,\langle \tilde{\sigma}\rangle)=\zz/2\zz$. 
\end{proof}

\begin{lemma}  \label{autP:3-19}
There exists an $X$ in \textnumero 3.19 such that $\autp(X)=\wg_X \cong \zz/2\zz$. 
\end{lemma}
\begin{proof} 
If $X$ is in \textnumero 3.19, it is the blow-up of $Q$ in two non-collinear points, say $p$ and $q$. Let $p = [1:0:0:0:0]$ and $q = [0:1:0:0:0]$ be two points on $Q$. Let $Q$ be a smooth quadric given by $x_0x_1 + x_2^2 + x_3^2 + x_4^2 =0$ that is blown-up at $p$ and $q$ to get $X$. Then the involution
\[
\sigma: [x_0:x_1:x_2:x_3:x_4] \mapsto [x_1:x_0:x_2:x_3:x_4].
\]
in $\aut(\pp^4;Q)$ swaps points $p$ and $q$.
By Lemma \ref{Lemma:autblowup}, $\sigma$ lifts to an automorphism $\tilde{\sigma}\in \aut(X)$ such that it swaps the two exceptional divisors, $E_1$ and $E_2$. Thus, we get $\autp(X,\langle \tilde{\sigma}\rangle)=\zz/2\zz$.
\end{proof}

\begin{lemma}  \label{autP:3-20}
There exists an $X$ in \textnumero 3.20 such that $\autp(X)=\wg_X \cong \zz/2\zz$. 
\end{lemma}
\begin{proof} 
If $X$ is in \textnumero 3.20, it is the blow-up of $Q$ in the disjoint union of two lines. Consider the four points 
\[
p=[1:0:0:0:0], \; \; q=[0:1:0:0:0], \; \; r = [0:0:1:0:0],\;\; s= [0:0:0:1:0]
\]
in $\mathbb{P}^4$. Define two lines $\ell$ and $m$ by $\ell := \overline{pq}$, $m := \overline{rs}$, 
that is, $\ell$ is the line through $p$ and $q$, and $m$ is the line through $r$ and $s$.
Suppose that $Q \subset \pp^4$ is a smooth quadric defined by $x_4^2 + x_0x_1 + x_2x_3 =0$, then it contains lines $\ell$ and $m$. Moreover, $\ell$ and $m$ are interchanged under the involution $\sigma \in \aut(Q;\ell\cup m)$
\[
\sigma: [x_0:x_1:x_2:x_3:x_4] \mapsto [x_2:x_3:x_0:x_1:x_4].
\]
So from Lemma \ref{Lemma:autblowup}, $\sigma$ lifts to an automorphism $\tilde{\sigma} \in \aut(X)$ that acts on $\operatorname{Pic}(X)$ by swapping $E_1$ and $E_2$. Thus we have $\autp(X,\langle \tilde{\sigma}\rangle)=\zz/2\zz$. 
\end{proof} 

\subsection{Blow-ups of \texorpdfstring{$\mathbb{P}^1 \times \mathbb{P}^1 \times \mathbb{P}^1$}{P1P1P1}}   \label{im(alpha)for(P1)3}
There are four families of Fano threefolds that contain a smooth member $X$ obtained as the blow-up of $\mathbb{P}^1 \times \mathbb{P}^1 \times \mathbb{P}^1$ along a curve $C$:
\begin{center}
\textnumero 4.3, \; \textnumero 4.6, \; \textnumero 4.8, \; \textnumero 4.13.
\end{center}
Unless mentioned otherwise, we use the Mori-Mukai description for a smooth member of these families from \cite[Table 4]{MM81classtable}, \cite{MMerratum}. 
Let $([x_0,x_1],[y_0,y_1],[z_0,z_1])$ be coordinates on $\pp^1\times\pp^1\times\pp^1$. Let $H_i$ be the pullback to $(\pp^1)^3$ of the hyperplane on $i$th copy of $\pp^1$ for $1\leq i\leq 3$. Let $\pi:X \to (\pp^1)^3$ be the blow-up map, where $E$ is the exceptional divisor. Then $\pic(X)=\zz[\pi^*H_1]\oplus \zz[\pi^*H_2]\oplus\zz[\pi^*H_3]\oplus\zz[E]$. 

\begin{lemma}  \label{autP:4.3}
There exists an $X$ in \textnumero 4.3 such that $\autp(X)=\wg_X \cong \zz/2\zz$. 
\end{lemma}
\begin{proof} If $X$ is in \textnumero 4.3, it is the blow-up of $\mathbb{P}^1\times \mathbb{P}^1 \times \mathbb{P}^1$ along a curve $C$ of tridegree $(1,1,2)$.  Consider a parametrization of $C$ given by the image of the embedding $\pp^1_{s_0,s_1} \hookrightarrow \pp^1\times\pp^1\times\pp^1$
\[
[s_0:s_1]   \longmapsto ([s_0:s_1],[s_0:s_1],[s_0^2:s_1^2])
\]
or equivalently, it is the complete intersection given by 
\[
C := \{ x_0y_1-x_1y_0 =0,\;z_0x_1y_1-z_1x_0y_0 =0 \}. 
\]
The curve $C$ is both smooth and irreducible, which can also be seen by Remark  \ref{rmk:magmause}. Let $\sigma \in \aut((\pp^1)^3)$ be the involution
\[
\sigma : ([x_0:x_1],[y_0:y_1],[z_0:z_1]) \mapsto ([y_0:y_1],[x_0:x_1],[z_0:z_1]).
\]
Then $C$ is $\sigma$-invariant. By Lemma \ref{Lemma:autblowup}, $\sigma$ lifts to a $\tilde{\sigma}\in \aut(X)$ such that $\tilde{\sigma}$ acts on $\pic(X)$ by swapping $\pi^*H_1$ and $\pi^*H_2$. Thus, $\autp(X,\langle \tilde{\sigma}\rangle)= \zz/2\zz$.
\end{proof}

\begin{lemma}  \label{autP:4.6}
There exists an $X$ in \textnumero 4.6 such that $\autp(X)=\wg_X \cong S_3$. 
\end{lemma}
\begin{proof}
If $X$ is in \textnumero 4.6, it is the blow-up of $\mathbb{P}^1 \times \mathbb{P}^1 \times \mathbb{P}^1$ in the tridiagonal curve $C$. Consider a parametrization of $C$ given by the image of the embedding $\mathbb{P}^1_{s_0,s_1} \hookrightarrow \mathbb{P}^1 \times \mathbb{P}^1\times \mathbb{P}^1$ 
\[
[s_0:s_1] \longmapsto ([s_0:s_1],[s_0:s_1],[s_0:s_1]),  
\]
that is, 
\[
C :=\{ x_0y_1-y_0x_1 =0, \; y_0z_1-y_1z_0 =0, \; x_0z_1 - x_1z_0 =0\}.
\]
Let $\sigma, \tau \in \aut((\pp^1)^3)$
\begin{align*}
\sigma: ([x_0:x_1],[y_0:y_1],[z_0:z_1]) \mapsto ([y_0:y_1],[z_0:z_1],[x_0:x_1]),\\
\tau: ([x_0:x_1],[y_0:y_1],[z_0:z_1]) \mapsto ([y_0:y_1],[x_0:x_1],[z_0:z_1]).
\end{align*}
be automorphisms of orders $3$ and $2$, respectively. Note that $\sigma , \tau \in \aut((\pp^1)^3;C)$. From Lemma \ref{Lemma:autblowup}, they lift to automorphisms $\tilde{\sigma},\; \tilde{\tau} \in \aut(X;E)$. Let $\mathcal{G}:=\langle \tilde{\sigma}, \tilde{\tau} \rangle $, it acts on $\pic(X)$ by permuting the generators $\pi^*H_1,\pi^*H_2,\pi^*H_3$. Thus, $\autp(X,\mathcal{G})=S_3$.  
\end{proof}

\begin{lemma}  \label{autP:4.8}
There exists an $X$ in \textnumero 4.8 such that $\autp(X)=\wg_X \cong \zz/2\zz$. 
\end{lemma}
\begin{proof}
If $X$ is in \textnumero 4.8, it is the blow-up of $\mathbb{P}^1 \times \mathbb{P}^1 \times \mathbb{P}^1$ in a curve $C$ of tridegree $(0,1,1)$. We recall the following construction from \cite[\S 3.7]{calabi}-- let $\operatorname{pr}_1 : \pp^1\times\pp^1\times\pp^1 \to \pp^1$ be the projection onto the first factor. Let $H_1 := \operatorname{pr}_1^{-1}(\pp^1)$ and $\mathcal{C}$ be a curve of bidegree $(1,1)$ in $H_1 = \pp^1\times\pp^1$. Then $X$ can be obtained as the blow-up of $\pp^1\times\pp^1\times\pp^1$ in $\mathcal{C}$.
Assume that
\[
H_1 := \{x_0=0\}   \text{  and  } \mathcal{C}:= \{ y_0z_1-y_1z_0=0 \}.
\]
Then $C:= H_1 \cap \mathcal{C}$ is both smooth and irreducible. Note that $C$ is invariant under the involution $\sigma \in \aut((\pp^1)^3)$
\[
\sigma : ([x_0:x_1],[y_0:y_1],[z_0:z_1]) \mapsto ([x_0:x_1],[z_0:z_1],[y_0:y_1])
\]
which acts on $\pic((\pp^1)^3)$ by swapping two of its generators. Since $\pic(X) \cong \pic((\pp^1)^3)\\
\oplus\zz[E]$, $\sigma$ also induces an involution on $\pic(X)$. If we write $\tilde{\sigma}$ for a lift of $\sigma$ to $X$, then we obtain $\autp(X, \langle \tilde{\sigma}\rangle )=\zz/2\zz$. 
\end{proof}

\begin{lemma}  \label{autP:4.13}
There exists an $X$ in \textnumero 4.13 such that $\autp(X)=\wg_X \cong \zz/2\zz$. 
\end{lemma}
\begin{proof} 
If $X$ is in \textnumero 4.13, it can be obtained as the blow-up of $\mathbb{P}^1 \times \mathbb{P}^1 \times \mathbb{P}^1$ in a curve $C$ of tridegree $(1,1,3)$. Consider a parametrization of $C$ given by the image of the embedding $\pp^1_{s_0,s_1} \hookrightarrow \pp^1\times\pp^1\times\pp^1$
\[
[s_0:s_1]   \longmapsto ([s_0:s_1],[s_0:s_1],[s_0^3:s_1^3])
\]
or equivalently, by the complete intersection 
\[
C:= \{x_0y_1-x_1y_0=0,\; z_0y_1^2x_1-z_1y_0^2x_0 =0\}.
\]
This intersection defines a smooth and irreducible curve (Remark \ref{rmk:magmause}). The curve $C$ is also invariant under the involution $\sigma \in \aut((\pp^1)^3)$
\[
\sigma : ([x_0:x_1],[y_0:y_1],[z_0:z_1]) \mapsto ([y_0:y_1],[x_0:x_1],[z_0:z_1]). 
\]
which acts on $\pic((\pp^1)^3)$ by swapping two of its generators. Thus, $\sigma$ induces an action on $\pic(X)$ by an involution. Let $\tilde{\sigma}\in \aut(X,E)$ denote a lift of $\sigma$, then we have $\autp(X, \langle \tilde{\sigma}\rangle )=\zz/2\zz$.
\end{proof}

\subsection{Blow-ups of \textnumero 2.32} 
We recall that a smooth member $W$ of the family \textnumero $2.32$ is a divisor of bidegree $(1,1)$ on $\mathbb{P}^2\times \mathbb{P}^2$. Let us discuss the deformation families 
\begin{center}
\textnumero 3.7, \; \textnumero 3.13, \; \textnumero 4.7.
\end{center}

If $X$ is a smooth member of \textnumero 3.7 or \textnumero 3.13, then there is a curve $C \subset W$ such that $X$ is the blow-up of $W$ in $C$, \cite[Table 3]{MM81classtable}. Let $\pi : X \to W$ be the blow-up map, and let $\operatorname{pr}_i : \pp^2\times\pp^2 \to \pp^2$ denote the projection onto the $i$-th factor, for $i=1,2$. Set $H_i:= \big(\operatorname{pr}_i^*\mathcal{O}_{\pp^2}(1)\big)|_W$. Then $\pic(X)=\zz[\pi^*H_1]\oplus\zz[\pi^*H_2] \oplus\zz[E]$, where $E$ is the exceptional divisor. 

We assign the projective coordinates $([x_0:x_1:x_2],[y_0:y_1:y_2])$ to $\pp^2\times\pp^2$. 

\begin{lemma}  \label{autP:3.7}
There exists an $X$ in \textnumero 3.7 such that $\autp(X)=\wg_X \cong \zz/2\zz$. 
\end{lemma}
\begin{proof} 
If $X$ is in family \textnumero 3.7, the center of the blow-up is a curve $C$, which can be described as the intersection $D_1\cap D_2$ of two divisors $D_1, \; D_2 \in |-\frac{1}{2} K_W|$. Note that $-\frac{1}{2} K_W \sim H_1 +H_2$. Recall the threefold $W$ from \eqref{eqn:W}, defined by $x_0y_0 + x_1y_1 + x_2y_2=0$, and the involution $\sigma \in \aut(W)$
\[
\sigma : ([x_0:x_1:x_2],[y_0:y_1:y_2]) \mapsto ([y_0:y_1:y_2],[x_0:x_1:x_2]),
\]
as in \eqref{eqn:involutionforW}. 
We also recall that $\autp(W,\langle \sigma\rangle)=\zz/2\zz$ (Lemma \ref{autP:2-32}). Now suppose that 
\[
D_1 := \{\, x^\top A y = 0 \,\}, \qquad 
D_2 := \{\, x^\top B y = 0 \,\},
\]
where $A=(a_{ij})$ and $B=(b_{ij})$ are symmetric $3\times 3$ matrices, $x = (x_0,x_1,x_2)^\top$ and $y = (y_0,y_1,y_2)^\top$. This makes both $D_1$ and $D_2$ invariant under $\sigma$. Using Remark \ref{rmk:magmause}, there exist matrices $A$ and $B$ so that $C:= D_1 \cap D_2 \cap W$ is a smooth and irreducible curve (on $W$). The curve $C$ is also $\sigma$-invariant. Finally, if $\tilde{\sigma}\in \aut(X;E)$ is a lift of $\sigma$, then $\tilde{\sigma}$ acts on $\pic(X)$ by swapping generators $\pi^*H_1$ and $\pi^*H_2$. Thus, we obtain $\autp(X,\langle \tilde{\sigma}\rangle) = \mathbb{Z}/2\mathbb{Z}$.
\end{proof}

Let $X$ be a smooth Fano threefold in family \textnumero 3.13. In \cite[III]{Mat95}, it was shown that $\wg_X\cong \zz/2\zz$, but we shall see that in fact $\wg_X \cong S_3$. 

\begin{lemma}  \label{autP:3.13}
There exists an $X$ in \textnumero 3.13 such that $\autp(X)=\wg_X \cong S_3$.
\end{lemma}
\begin{proof}
If $X$ is in family \textnumero 3.13, the center of the blow-up is a curve $C$ of bidegree $(2,2)$ and
\[
C \hookrightarrow W \hookrightarrow \mathbb{P}^2 \times \mathbb{P}^2 \xrightarrow{\operatorname{pr}_i} \mathbb{P}^2
\]
is an embedding for $i=1,2$. We now specialize to the following member of family \textnumero 3.13. By \cite[\S 5.19]{calabi}, there exists a smooth Fano threefold $X$ in this family which, in suitable coordinates $([x_0:x_1:x_2],[y_0:y_1:y_2],[z_0:z_1:z_2])$ on $\pp^2\times\pp^2\times\pp^2$, is given by
\[
\begin{cases}
    x_0y_0 + x_1y_1+x_2y_2=0\\
    y_0z_0+ y_1z_1 + y_2z_2 =0\\
    x_0z_1+ x_1z_0 + x_1z_2 - x_2z_1 -2x_2z_2 =0,
\end{cases}
\]
For this member, $\aut(X)\cong \mathbb{G}_a \rtimes S_3$. Here the subgroup $S_3$ is generated by the involutions 
\begin{align*}
\tau_{x,y}: ([x_0:x_1:x_2],&[y_0:y_1:y_2],[z_0:z_1:z_2]) \mapsto 
([y_0+2y_1+\\
&y_2:2y_0:y_0+y_2],[x_1:x_0-x_2:2x_2-x_1],[z_0:z_1:-z_2]),\\
\tau_{x,z}: .([x_0:x_1:x_2]&,[y_0:y_1:y_2],[z_0:z_1:z_2]) \mapsto \\
&([z_0:z_1:-z_2],[y_0:y_1:-y_2],[x_0:x_1:-x_2])
\end{align*} 
such that its action on $\pic(\pp^2\times\pp^2\times\pp^2)$ coincides with the permutation action of $S_3$. From Proposition \ref{thm:lefschetz}, we know $\pic(X)\cong \pic(\pp^2\times\pp^2\times\pp^2)$. Thus, we have $\autp(X,G)= S_3$ where $G$ is generated by $\tau_{x,y},\tau_{x,z}$.  
\end{proof}

\begin{lemma}  \label{autP:4.7}
There exists an $X$ in \textnumero 4.7 such that $\autp(X)=\wg_X \cong \zz/2\zz$. 
\end{lemma}
\begin{proof} 
If $X$ is in family \textnumero 4.7, the center of the blow-up is a curve $C$ which is given by the disjoint union $C_1\sqcup\; C_2$ of curves $C_1$ and $C_2$ of degrees $(0,1)$ and $(1,0)$, respectively. Recall the Fano threefold $W$ from \eqref{eqn:W}, defined by $x_0y_0 + x_1y_1 + x_2y_2=0$, and the involution $\sigma \in \aut(W)$
\[
\sigma : ([x_0:x_1:x_2],[y_0:y_1:y_2]) \mapsto ([y_0:y_1:y_2],[x_0:x_1:x_2]),
\]
as in \eqref{eqn:involutionforW}.
Assume that
\begin{align*}
    C_1:= \{y_0 = y_1 =0\},    \hspace{0.5cm} C_2:= \{x_0 = x_1 =0\}.
\end{align*}
Then $\sigma$ also swaps the curves $C_1$ and $C_2$ but keeps $C_1\sqcup C_2$ invariant. Note that $C_1 \sqcup C_2$ is smooth and irreducible on $W$ since each curve $C_i$ is, and $\sigma \in \aut(W;C_1\sqcup C_2)$. Let $E_i$ denote the exceptional divisor corresponding to $C_i$, for $i=1,2$. Then $\pic(X) \cong \zz[\pi^*H_1]\oplus \zz[\pi^*H_2]\oplus \zz[E_1]\oplus \zz[E_2]$. 

Now if $\tilde{\sigma} \in \aut(X,E_1\cup E_2)$ is a lift of $\sigma$, then it acts on $\pic(X)$ by swapping $E_1$ with $E_2$ and $\pi^*H_1$ with $\pi^*H_2$. This shows that $\autp(X,\langle \tilde{\sigma}\rangle)=\mathbb{Z}/2\mathbb{Z}$.
\end{proof}

\subsection{Blow-up of cones} We consider families \textnumero 3.9 and \textnumero 4.2. A smooth member $X$ of these families is obtained by blowing up the Veronese cone and the quadric cone, respectively. 

We recall the construction of Fano threefolds in families \textnumero 3.9 and \textnumero 4.2 from \cite[\S 4.6]{calabi}. Let $\mathscr{S}:=\pp^2$ for family \textnumero 3.9 and let $\mathscr{S}:=\pp^1\times \pp^1$ for family \textnumero 4.2. Let $C\subset \mathscr{S}$ be a smooth quartic curve in the first case, and a smooth curve of bidegree $(2,2)$ in the second case. Consider the involution $\iota \in \aut(\pp^1\times \mathscr{S})$ 
\[
([u:v],p) \longmapsto ([v:u],p).
\]
The construction detailed in \cite{calabi}, starting with a double cover of $\pp^1\times \mathscr{S}$ followed by a suitable resolution and contraction, yields a smooth Fano threefold $X$ together with two birational contractions 
\[
\phi: X \to V \qquad \phi' : X \to V'.
\]
Let $F$ and $F'$ denote their exceptional divisors, respectively. The involution $\iota$ lifts to an involution of $X$ that swaps $F$ and $F'$, thereby exchanges the contractions $\phi$ and $\phi'$. Moreover, we have $V \cong V'$, but $\iota$ does not descend to an automorphism of either $V$ or $V'$.

If $X$ belongs to family \textnumero 3.9, then $V$ belongs to family \textnumero 2.36, whereas if $X$ belongs to family \textnumero 4.2, then $V$ belongs to family \textnumero 3.31. 
\medskip

Suppose first that $X$ is a Fano threefold in \textnumero 3.9. Then $ V:=\pp(\mathcal{O}\oplus \mathcal{O}(2))$. Let $\operatorname{pr}: V \to \pp^2$ be the natural projection. The center $C$ of the blow-up map $\phi$ is contained in the positive section, $S_+$ of $\operatorname{pr}$, and is identified with the quartic curve in $\pp^2$. 

In Matsuki's notation following \cite[III]{Mat95}, let $\ell_1, \dots, \ell_4$ be the extremal generators of $\overline{NE}(X)$. Let $E_i$ denote the exceptional divisor corresponding to the extremal contraction, $\operatorname{cont}_{\ell_i}$ of $\ell_i$, for $1\leq i\leq 4$. In this notation, we identify $cont_{\ell_1}$ with $\phi$ and $E_1$ and $F$. 

Set $H:=\operatorname{pr}^*\mathcal{O}_{\pp^2}(1)$. Then  $\pic(X)=\zz[\operatorname{cont}_{\ell_1}^*H, E_1,E_4]$.

\begin{lemma}    \label{autP:3.9}
There exists an $X$ in \textnumero 3.9 such that $\autp(X)=\operatorname{WG}_X \cong \zz/2\zz$. 
\end{lemma}
\begin{proof}
The involution $\iota \in \aut(X)$ acts on $\overline{NE}(X)$ by the permutation $(13)(24)$, where the cycle $(i\;j)$ denotes the interchange of rays $\ell_i$ and $\ell_j$. This induces the following action on $\pic(X)$:
\[
\iota^*\colon
\left\{
\begin{aligned}
\operatorname{cont}_{\ell_1}^*H
    &\longmapsto \operatorname{cont}_{\ell_1}^*H,\\
E_1 &\longmapsto 4\operatorname{cont}_{\ell_1}^*H - E_1,\\
E_4 &\longmapsto 2\operatorname{cont}_{\ell_1}^*H - E_1 + E_4.
\end{aligned}
\right.
\]
In particular, $\iota^*$ is a nontrivial involution of $\pic(X)$. Thus, we get $\autp(X,\langle \iota\rangle)=\zz/2\zz$.
\end{proof}

Now suppose that $X$ is a Fano threefold belonging to family \textnumero 4.2. Then $ V:=\pp(\mathcal{O}\oplus \mathcal{O}(1,1))$ with the natural projection $\operatorname{pr}:V \to \pp^1\times \pp^1$. Let $S_-$ and $S_+$ denote two natural sections of $\operatorname{pr}$. The center $C$ of the blow-up map $\phi: X \to V$ is an ellptic curve $C \subset S_+$, and it is identified with a smooth divisor of bidegree $(2,2)$ on $\pp^1\times\pp^1$. 

Set $H_1=\operatorname{pr}^*\mathcal{O}(1,0)$ and $H_2:=\operatorname{pr}^*\mathcal{O}(0,1)$. By \cite{Mat95,Mat23}, let $\ell_1, \dots , \ell_6$ denote the extremal generators of $\overline{NE}(X)$, and let $E_i$ denote the exceptional divisor of the extremal contraction corresponding to ray $\ell_i$, for $1\leq i \leq 6$. In this notation, we may identify $\phi$ with $\operatorname{cont}_{\ell_1}$ and $F$ with $E_1$. Then 
\[
\pic(X)= \zz[\phi^*H_1, \phi^*H_2, E_5, \widetilde{S_+} ],
\]
where $E_5=\phi^*S_-$ and $\widetilde{S_+}$ is the strict transform of $S^+$. 

\begin{lemma}   \label{autP:4.2}
There exists an $X$ in \textnumero 4.2 such that $\autp(X)=\operatorname{WG}_X\cong (\zz/2\zz)^2$. 
\end{lemma}
\begin{proof}
Consider the involution $\sigma \in \aut(\pp^1\times\pp^1)$ 
  \begin{align*}
   \sigma: ([x_0:x_1],[y_0:y_1])  &\mapsto ([y_0:y_1],[x_0:x_1]),
\end{align*}
and the hypersurface of bidegree $(2,2)$ in $\pp^1\times\pp^1$ defined by 
\[
a(x_0^2y_0^2+x_1^2y_1^2) + b(x_1^2y_0^2+x_0^2y_1^2) + cx_0x_1y_0y_1 =0.
\]
The equation is $\sigma$-invariant. For general choices of coefficients $a,b,c$, it is smooth and irreducible (Remark \ref{rmk:magmause}), and hence, defines an elliptic curve $C$. Observe that $\sigma$ acts on $\pic(\pp^1\times\pp^1)$ by interchanging $\mathcal{O}(1,0)$ and $\mathcal{O}(0,1)$, thereby $\sigma^*(\mathcal{O}\oplus \mathcal{O}(1,1))=\mathcal{O}\oplus \mathcal{O}(1,1)$, and $\sigma$ lifts to an automorphism of $V$, preserving sections $S_-$ and $S_+$. Furthermore, $C\subset S_+$ is $\sigma$-invariant, so $\sigma$ lifts to an involution $\widetilde{\sigma}\in \aut(X)$. Observe that $\sigma$ acts on $\neo(X)$ by the permutation $(34)(56)$, where the cycle $(i\;j)$ denotes the interchange of rays $\ell_i$ and $\ell_j$. 

Consequently, $\widetilde{\sigma}^*$ induces the following action on $\pic(X)$
\[
\widetilde\sigma^*\colon
\left\{
\begin{aligned}
\phi^*H_1&\longmapsto\phi^*H_2,\\
\phi^*H_2&\longmapsto\phi^*H_1,\\
E_5&\longmapsto E_5,\\
\widetilde S^+&\longmapsto\widetilde S^+.
\end{aligned}
\right.
\]
Thus, we get $\autp(X,\langle \widetilde{\sigma}\rangle)\cong \zz/2\zz$. 

From the definition of $\iota$ above, it acts trivially on $\pp^1\times\pp^1$. With Matsuki's labeling, the action of $\iota$ on $\overline{NE}(X)$ is given by the permutation $(12)(35)(46)$. The induced action on $\pic(X)$ is as follows
\[
\iota^*\colon
\left\{
\begin{aligned}
\phi^*H_1&\longmapsto\phi^*H_1,\\
\phi^*H_2&\longmapsto\phi^*H_2,\\
E_5&\longmapsto\widetilde S^+,\\
\widetilde S^+&\longmapsto E_5.
\end{aligned}
\right.
\]
So, $\iota^* \in \aut(\pic(X))$ is a nontrivial involution distinct from $\sigma^*$, and $\autp(X,\langle \iota\rangle)=\zz/2\zz$. 


Furthermore, $\widetilde{\sigma}\iota=\iota\widetilde{\sigma}$. Therefore, for $\mathcal{G}:=\langle \widetilde{\sigma}, \iota\rangle \cong (\zz/2\zz)^2$, we obtain 
\[
\autp(X,\mathcal{G}) = (\zz/2\zz)^2. 
\]
\end{proof}

\subsection{Blow-up of a Fano threefold from \textnumero 3.19} 
We recall from Lemma \ref{autP:3-19} that a Fano threefold $Y$ in family \textnumero 3.19 is the blow-up of a quadric threefold $Q$ at two points $p$ and $q$. Let $\pi_1 : Y \to Q$ be the blow-up map. If $X$ is a Fano threefold in family \textnumero 4.4, then Mori-Mukai describe $X$ as the blow-up of $Y$ along the strict transform of a conic $C\subset Q$ passing through the two points. Let $\pi_2 : X \to Y$ denote this blow-up map. 

Set $H:=\pi_2^*\pi_1^*\mathcal{O}_Q(1)$. Then $\pic(X) = \zz[H]\oplus \zz[E_1]\oplus \zz[E_2]\oplus \zz[F]$, where $E_1$ and $E_2$ are the pullbacks to $X$ of the exceptional divisors of the points $p,q \in Q$, and $F$ is the exceptional divisor over the strict transform of $C$.

From Lemma \ref{autP:3-19}, we recall the smooth quadric $Q: x_0x_1 + x_2^2 + x_3^2 + x_4^2 =0 $, points $p=[1:0:0:0:0]$ and $ q=[0:1:0:0:0]$ on $Q$, and $\sigma\in \operatorname{Aut}(Q;{p,q})$ 
\[
\sigma: [x_0:x_1:x_2:x_3:x_4] \mapsto [x_1:x_0:x_2:x_3:x_4].
\] 
\begin{lemma}  \label{autP:4.4}
There exists an $X$ in \textnumero 4.4 such that $\autp(X)=\wg_X \cong \zz/2\zz$. 
\end{lemma}
\begin{proof}
Suppose that $C$ is given by 
\[
C:= \{x_2 = x_3 =0,\; x_0x_1 + x_2^2 + x_3^2 + x_4^2 =0\} \subset Q,
\]
it passes through the points $p$ and $q$. The curve is smooth and irreducible (Remark \ref{rmk:magmause}). Observe that since $\sigma$ swaps points $p$ and $q$, its lift to $\aut(Y)$ swaps the exceptional divisors $\pi_1^{-1}(p)$ and $\pi_1^{-1}(q)$ and thus, induces an involution on $\pic(Y)$. Since $C$ is also $\sigma$-invariant, this further lifts to a $\tilde{\sigma} \in \aut(X;F)$ such that the action on $\pic(X)$ is by swapping $E_1$ and $E_2$. This shows that $\autp(X,\langle \tilde{\sigma}\rangle) =\mathbb{Z}/2\mathbb{Z}$. 
\end{proof}

\subsection{Blow-up of a Fano threefold from \textnumero 2.29}  \label{im(alpha)for5-1}
Let $X$ denote a smooth Fano threefold in the family \textnumero 5.1. We recall the Mori-Mukai construction of $X$. Let $C$ be a smooth conic on a quadric $Q \subset \pp^4$, and let $p_1, p_2, p_3\in C$ be three distinct points. Then $X$ can be constructed as follows: first blow-up $Q$ along $C$, obtaining 
$\pi_1 : Y \to Q$, where $Y$ belongs to family \textnumero 2.29. 
Next, blow up $Y$ along the three lines $e_1, e_2, e_3$, 
where each $e_i$ is the exceptional line on $Y$ contracted by $\pi_1$ to the point $p_i \in Q$.

By \cite[Corollary 2.6]{MM86}, this is equivalent to the following construction: Let $f : \widetilde{Y} \to Q$ be the blow-up of $Q$ at the points $p_1, p_2, p_3$, 
and let $\widetilde{C}$ denote the strict transform of the conic $C$ under $f$. 
Then $X$ is obtained as the blow-up $\pi : X \to \widetilde{Y}$ along $\widetilde{C}$.
 We have
\[
\pic(X) = \pi^*\pic(\widetilde{Y}) \oplus \zz[E_7], \hspace{0.2cm} \pic(\widetilde{Y}) = \zz[f^*H] \oplus \zz[E_1]\oplus\zz[E_2]\oplus\zz[E_3]. 
\]
where $E_i$ is the exceptional divisor on $\widetilde{Y}$ that $f$ contracts to $p_i$, $1\leq i \leq 3$ and $E_7$ is the exceptional divisor of the blow-up $\pi$. 

\begin{lemma}  \label{autP:5.1}
There exists an $X$ in \textnumero 5.1 such that $\autp(X)=\wg_X \cong S_3$. 
\end{lemma}
\begin{proof}
Let $Q:=\{ x_0x_1+x_2^2+x_3^2+x_4^2=0\}$ and $C := \{ x_0 = x_1 =0, x_2^2+x_3^2+x_4^2=0\}$. Let 
\[
p_1 = [0:0:1:\omega:\omega^2],\; p_2 = [0:0:\omega:1:\omega^2], \; p_3 = [0:0:1:\omega^2:\omega] \; \in C,
\]
where $\omega$ is a primitive third root of unity. Consider the involutions of $\pp^4$
\begin{align*}
\sigma &: [x_0:x_1:x_2:x_3:x_4] \mapsto [x_0:x_1:x_3:x_2:x_4] \\
\tau &: [x_0:x_1:x_2:x_3:x_4] \mapsto [x_0:x_1:x_4:x_3:x_2].
\end{align*}
such that $\sigma, \tau \in \aut(Q;C)$ and $\langle\sigma,\tau\rangle = S_3$. 

Note that $\sigma$ swaps $p_1$ and $p_2$, and lifts to an automorphism of $\widetilde{Y}$ that swaps the respective exceptional divisors, $E_1$ and $E_2$. Hence $\sigma$ induces an involution on $\pic(\widetilde{Y})$ and consequently on $\pic(X)$. We denote a lift of $\sigma$ to $\aut(X;E_7)$ by $\widetilde{\sigma}$. Similarly, $\tau$ swaps $p_1$ and $p_3$ and lifts to a $\widetilde{\tau} \in \aut(X;E_7)$ that induces an involution on $\pic(X)$. This implies that 
\[
\autp(X,\mathcal{G}) = S_3
\] 
where $\mathcal{G}:=\langle \widetilde{\sigma}, \widetilde{\tau}\rangle \cong S_3$ is a subgroup of $\aut(X)$.
\end{proof}

\subsection{Toric Fano threefolds}  \label{subsec:toric}
We consider the toric Fano threefolds with nontrivial $\wg_X$; see \cite{batyreveng} for the classification of smooth toric Fano threefolds. We note that the families \textnumero 3.25 and \textnumero 3.27
are toric, but are omitted from the discussion below since they are treated separately in Lemma \ref{autP:3-25} and Example \ref{autP:3.27}, respectively. It remains to consider the families
\begin{center}
    \textnumero 3.31,\; \textnumero 4.10,\; \textnumero 4.12,\;
    \textnumero 5.2,\; \textnumero 5.3.
\end{center}
We can use the \emph{NormalToricVarieties} module in Macaulay2 to get a fan for each toric Fano threefold (\cite{NormalToricVarietiesSource}). 
For families  \textnumero 4.12, \textnumero 5.2, we use the grading given in
\cite[\S\S 97, 99]{quantum}. Note that the Mori-Mukai family \textnumero 4.12 is labeled as \textnumero 4.13 in \cite{quantum}.

Recall from Table \ref{weylfano}, $\wg_X =\zz/2\zz$ for all these families except for \textnumero 5.3. If $X$ is in \textnumero 5.3, we have $\wg_X=\zz/2\zz \times S_3$. It is possible to realize the upper bound $\autp(X)=\wg_X$ in all these cases using the grading of Cox coordinates as follows. 

\begin{lemma}  \label{autp:3-31,4-10,4-12,5-2}
    Let $X$ be a smooth Fano threefold in \textnumero 3.31, \textnumero 4.10, \textnumero 4.12 or \textnumero 5.2. Then $\autp(X)=\zz/2\zz$. 
\end{lemma}
\begin{proof} 
For each family, we list a pair $(\sigma, M)$, where $\sigma$ is a permutation on Cox coordinates of $X$ and $M_{\sigma}$ is an automorphism on the Picard lattice induced by $\sigma$. Then to see if $\sigma$ indeed lifts to an automorphism in $\aut(X)$, it suffices to check that $\sigma$ preserves the irrelevant ideal of Cox ring, or equivalently, the set of maximal cones of $X$ \cite{Cox95}. 

If $X$ is in \textnumero 3.31, the weighted grading is
\begin{center}
\begin{tabular}{c c c c c c }
   $s_0$ & $s_1$ & $t_0$ & $t_1$ & $x$ & $y$ \\
    \hline
   1  & 1 & 0& 0 & -1 & 0  \\
   0  & 0 & 1 & 1 & -1  & 0 \\
   0  & 0 & 0 & 0 & 1 & 1 \\
\end{tabular}
\end{center}

Let $\pic(X) = \zz[[D_{s_0}] , [D_{t_0}], [D_y]]$ and $\sigma =  (13)(24)$, then 
\[
M_{\sigma}= \begin{pmatrix}
           0 & 1 & 0\\
           1 & 0 & 0 \\
           0  & 0 & 1
\end{pmatrix}. \]

If $X$ is in \textnumero 4.10, the weighted grading is
      \begin{center}
  \begin{tabular}{c c c c c c c }
  $u$ & $v$ & $s_0$  & $s_1$  & $t_0$ & $t_1$ & $w$\\
   \hline
1 & -1  & 0 & 0 & 0  & 0  &1  \\
0 & 0  & 0 & 0 & 1 &  -1  & 1  \\
0 & 0  & 1 & 1 & 0 &  0  &0  \\
0 & 1  & 0 & 0 & 0 &  1  & -1
   \end{tabular}
 \end{center}
 
Let $\pic(X) = \zz[ [D_u]$, $[D_v], [D_{s_0}], [D_{t_0}] ]$ and $\sigma = (15)(26) $, then 

\[
M_{\sigma}= \begin{pmatrix}
       0 & 0 & 0 & 1 \\ 
         1 & 1 & 0 & -1 \\ 
           0 & 0 & 1 & 0 \\ 
             1 & 0 & 0 & 0 
 \end{pmatrix}. \] 

If $X$ is in \textnumero 4.12, the weighted grading is 
\begin{center}
\begin{tabular}{c c c c c c c}
   $s_0$ & $s_1$ & $x$ & $y_2$ & $y_3$ & $u$ & $v$ \\
    \hline
   1  & 1 & -1& 0 & 0 & 0 & 0  \\
   0  & 0 & -1 & 0 & 0 & 1 & 1 \\
   0  & 0 & 1 & 1 & 0 & -1 & 0   \\
   0  & 0 & 1 & 0 & 1 & 0 &  -1  \\
\end{tabular}
\end{center}

Let $\pic(X) = \zz[[D_{s_0}], [D_x] , [D_{y_2}], [D_{y_3}]]$ and $\sigma =  (45)$, we have
\[ M_{\sigma} = \begin{pmatrix}
      1 & 0 & 0 & 0\\
       0 & 1 & 0 & 0\\
        0 & 0 & 0 & 1\\
         0 & 0 & 1 & 0
\end{pmatrix}. \]

Let us now assume that $X$ is in \textnumero 5.2 with weighted grading
\begin{center}
\begin{tabular}{c c c c c c c c}
   $s_0$ & $s_1$ & $t_2$ & $t_3$ & $x$ & $y$ & $u$ & $v$\\
   \hline
   1  & 1 & 0& 0 & -1 & 0 & 0 & 0 \\
   0  & 0 & 1 & 1 & 0  & -1 & 0 & 0 \\
   0  & 0 & 0 & 1 & 1 & 0 & -1  & 0 \\
   0  & 0 & 1 & 0 & 1 & 0 &  0  & -1  \\
   0  & 0 & -1 & -1 & 1 & 0 & 1 & 1
\end{tabular}
\end{center}
If $\pic(X)= \zz[[D_{s_0}], [D_{t_2}], [D_{t_3}], [D_x], [D_y]]$ and $\sigma = (34)$, then
\[ M_{\sigma}= \begin{pmatrix}
     1 & 0 & 0 & 0 & 0\\
       0 & 0 & 1 & 0 & 0\\
         0 & 1 & 0 & 0 & 0\\
           0 & 0 & 0 & 1 & 0\\
             0 & 0 & 0 & 0 & 1\\
\end{pmatrix}.\]
In each case, the permutation $\sigma$ sends a maximal cone to a maximal cone. So it lifts to an automorphism $\tilde{\sigma}$ of $X$, and we get $\autp(X,\langle \tilde{\sigma}\rangle)=\zz/2\zz$. 
\end{proof}

\begin{lemma}  \label{autp:5-3}
    Let $X$ be a smooth Fano threefold in \textnumero 5.3, then $\autp(X)=\wg_X \cong \zz/2\zz \times S_3.$
\end{lemma}
\begin{proof}
   If $X$ is a member of \textnumero 5.3, the weighted grading of Cox coordinates is
  \begin{center}
  \begin{tabular}{c c c c c c c c }
  $u$  & $v$  & $w$  & $x$ & $y_0$  & $y_1$  & $s_0$  & $s_1$\\
  \hline
1 & -1   &  1  &  0  & 0 &  0 &  0 & 0 \\
0  &  0  & 0  & 1  & -1  & 1 &  0 & 0\\
0  &  0  & 1   & -1 & 1 &  0  & 0  & 0 \\
0  & 0  & 0  & 0 &  0  & 0  &  1  & 1 \\ 
0   & 1   & -1  & 1 & 0  & 0  & 0 & 0
 \end{tabular}
 \end{center}

Assume that $\pic(X)= \zz[[D_u], [D_v], [D_w], [D_{y_1}],[D_{s_0}]]$. 
Let $g:= (14)(25)(36)$. It induces 
\[
M_{g} := \begin{pmatrix}
     1 & 0 & -1 & 1 & 0\\
    0 & 1 & 1 & -1 & 0\\
    0 & 0 & 0 & 1 & 0\\
    0 & 0 & 1 & 0 & 0\\
    0 & 0 & 0 & 0 & 1
\end{pmatrix}
\]
The permutation $\sigma:= (12)(36)(45)$ on Cox coordinates induces an action on $\pic(X)$ via
\[
M_{\sigma}:= \begin{pmatrix}
    0 & 1 & 0& 0 & 0\\
        1 & 0 & 0& 0 & 0\\
            0 & 0 & 0& 1 & 0\\
                0 & 0 & 1& 0 & 0\\
                    0 & 0 & 0& 0 & 1\\
\end{pmatrix}. 
\]
Now consider the automorphism $\tau := (153)(264)$ acts on $\pic(X)$ via the matrix 
\[
M_{\tau}:= \begin{pmatrix}
    0 & 1 & 1 & -1 & 0\\
    0 & 0 & 0 & 1 & 0\\
    1 & 0 & 0 & 0 & 0\\
    1 & 0 & -1 & 1 & 0\\
    0 & 0 & 0 & 0 & 1
\end{pmatrix}
\]
of order $3$. Note that $\langle \tau , \sigma\rangle \cong S_3$ and $g\not\in \langle \tau , \sigma\rangle $. Let $\mathcal{G}:= \langle g\rangle \times \langle \sigma, \tau\rangle$. The set of maximal cones is preserved under $\mathcal{G}$, so it lifts to a group $\widetilde{\mathcal{G}}\subset \aut(X)$ and we get 
$\autp(X,\widetilde{\mathcal{G}})=\zz/2\zz \times S_3$. 
\end{proof}

\subsection{Product of \texorpdfstring{$\mathbb{P}^1$}{P1} and a del Pezzo surface}   \label{autP:P1timesdP}
Let $X$ be a smooth Fano threefold of the form $\mathbb{P}^1 \times S_d$, where $S_d$ is a del Pezzo surface of degree $d$, $1 \leq d \leq 9$. Then $X$ is a member of one of the following $10$ deformation families 
\begin{center}
\textnumero 2.34,\; \textnumero 3.27, \; \textnumero 3.28, \; \textnumero 4.10, \; \textnumero 5.3, \; \textnumero 6.1,\; \textnumero 7.1,\; \textnumero 8.1,\; \textnumero 9.1,\; \textnumero 10.1.
\end{center}
If $X$ is in family \textnumero $2.34$, we have $d=9$ and it is isomorphic to $\mathbb{P}^1 \times \mathbb{P}^2$. If $X$ is in family \textnumero $3.28$, $d=8$ and $X$ is isomorphic to $\mathbb{P}^1 \times \operatorname{Bl}_p\mathbb{P}^2$. We forego the discussion of these two families since each has $WG_X=0$ from \cite{Mat23}.

If $X$ is in \textnumero $3.27$, we can see that $\autp(X,G)$ can be any subgroup of $S_3$ for a suitable choice of $G\subset \aut(X)$. 

\begin{example}  \label{autP:3.27}
\normalfont From \cite{MM81classtable}, a smooth Fano threefold in \textnumero 3.27 is isomorphic to $\pp^1 \times \pp^1\times \pp^1$. Let $\pic(\pp^1\times\pp^1\times\pp^1) = \zz[H_1]\oplus \zz[H_2]\oplus \zz[H_3]$, where $H_i := \operatorname{pr}_i^*\mathcal{O}_{\pp^1}(1)$ and $\operatorname{pr}_i : (\pp^1)^3 \to \pp^1$ are natural projections, for $1\leq i \leq 3$. 
Consider the automorphisms of  $\mathbb{P}^1_{x_0,x_1} \times \mathbb{P}^1_{y_0,y_1} \times \mathbb{P}^1_{z_0,z_1}$
\begin{align*}
    \sigma : ([x_0:x_1],[y_0:y_1],[z_0:z_1]) &\mapsto ([y_0:y_1],[z_0:z_1],[x_0:x_1])\\
    \tau : ([x_0:x_1],[y_0:y_1],[z_0:z_1]) &\mapsto ([y_0:y_1],[x_0:x_1],[z_0:z_1]
\end{align*}
They induce automorphisms of orders $3$ and $2$, respectively, on the Picard group. Finally, note that $S_3 \cong \langle \sigma,\tau\rangle$.
\end{example}

If $X$ is in \textnumero 4.10 or \textnumero 5.3, it has a toric description. See \S \ref{subsec:toric} for details. The remaining families of Fano threefolds have $\rho\geq 6$ and we have 
\[
\wg_X = W(R_{9-d}),
\]
where $d\leq 5$ is the degree of the del Pezzo surface and $R_{9-d}$ is the associated root system. The Weyl groups $W(R_{9-d})$ for $d\leq 4$ are much larger than those we considered earlier for Fano threefolds. The natural homomorphism 
\[
\aut(S_d) \rightarrow W(R_{9-d})
\]
is injective for del Pezzo surfaces with degree $d\leq 5$ (\cite[Lemma 6.2]{dolgaisko_planecremonagroups}). Moreover, the automorphism groups $\aut(S_d)$ have been classified in \cite[\S 6]{dolgaisko_planecremonagroups}. From the classification, one can compute $\autp(S_d)$, which is typically much smaller than $W(R_{9-d})$ for $d\leq 4$.

\section{First Group Cohomology of the Picard Group} \label{sec:H1}

Let $X$ denote a smooth Fano threefold, and let $G\subseteq \aut(X)$ be a finite group. Then $\pic(X)$ is a module over the the group ring $\zz[G]$. In this section, we compute the cohomology group $H^1(G,\pic(X))$. 

It is immediate that $H^1(G,\pic(X))=0$ for the families with $\autp(X)=0$, therefore we focus on the remaining families of Fano threefolds with nontrivial $\autp(X)$. First recall the following standard result about permutation modules.
\begin{lemma}\label{std-permutations-module-fact}
    Let $M\cong \zz^n$ be a lattice which has a basis permuted by a group $G$, then $H^1(H,M)=0$ for all subgroups $H$ of $G$.
\end{lemma}

The Mori cone $\neo(X)$ is generated by finitely many extremal rays $R_i$, where $R_i := \mathbb{R}_{\geq 0}[\ell_i]$ for some (extremal) rational curve $\ell_i$, for each $1\leq i\leq n$ \cite{mori1980threefolds}. For each family of Fano threefolds, the rational curves $\ell_i$ are described in \cite{Mat95, Mat23}. The author also gives a table containing the intersection products of curves $\ell_i$ with a $\zz$-basis of $\pic(X)$. 
 
\begin{lemma}\label{main-H1-result}
   Let $X$ be a Fano threefold with Picard rank $\rho\leq 5$. Then there exists a basis of $\pic(X)$ that is permuted by $\autp(X)$. 
\end{lemma}
\begin{proof} 
We may assume that $\autp(X)$ is nontrivial. Let $n$ denote the number of extremal generators of $\neo(X)$. Suppose that Picard rank $\rho=n$ and $X$ is a Fano threefold in one of the families
    \begin{center}
         \textnumero 2.6, \textnumero 2.32,  \textnumero 3.1,  \textnumero 3.3, \textnumero 3.17, \textnumero 3.27, \textnumero 4.1.   
    \end{center}
Then from the tables in \cite{Mat95,Mat23}, we observe that the $\zz$-basis of $\pic(X)$ given by Matsuki also generates the nef semigroup $\nef(X) \cap \pic(X)$. It is also straightforward to verify that each of these generators (except $-K_X$) is extremal, and they also span the Nef cone $\nef(X)$. Since the $G$-action permutes these generators, it induces a permutation action on $\pic(X)$. 
    
If $X$ belongs to one of the following families
    \begin{center}
        \textnumero 2.12, \textnumero 2.21,   \textnumero 3.7,  \textnumero 3.10,  \textnumero 3.13,  \textnumero 3.20,  \textnumero 3.25, \textnumero 3.31,  \textnumero 4.3,  \textnumero 4.6, \textnumero 4.7,  \textnumero 4.8,  \textnumero 4.12, 
    \end{center}
then we can find a set of extremal generators of $\nef(X) \cap \pic(X)$ from Matsuki's tables. 

In the second column of the table below, we list the extremal generators of \(\mathrm{Nef}(X) \cap \mathrm{Pic}(X)\) for a Fano threefold in each family. The third column lists a \(\mathbb{Z}\)-basis of \(\mathrm{Pic}(X)\); except when \(\rho = 2\), this basis coincides with that given by Matsuki.

\begin{table}[ht]
    \centering
    \renewcommand{\arraystretch}{1.3} 
    \setlength{\tabcolsep}{10pt} 
    \begin{tabular}{|c|c|c|}
        \hline
        MM \textnumero & Generators & $\pic(X)$ basis \\
        \hline
         $2.12$ & $3H - E,\; H$ & $3H - E,\; H$\\ \hline
        $2.21$ & $2H - E,\; H$ & $2H - E,\; H$\\ \hline
        $3.7$  & $H_1,\; H_2,\; H_1 + H_2 - E$ & $H_1,\; H_2,\; E$ \\ \hline
        $3.10$ & $H,\; H - E_1,\; H - E_2$ & $H,\; E_1,\; E_2$ \\ \hline
        $3.13$ & $H_1,\; H_2,\; H_1 + H_2 - E_1$ & $H_1,\; H_2,\; E_1$ \\ \hline
        $3.20$ & $H,\; H - E_1,\; H - E_2$ & $H,\; E_1,\; E_2$ \\ \hline
        $3.25$ & $H,\; H - E_1,\; H - E_2$ & $H,\; E_1,\; E_2$ \\ \hline
        $3.31$ & $E + H_1 + H_2,\; H_1,\; H_2$ & $H_1,\; H_2,\; E$ \\ \hline
         $4.3$  & $H_1,\; H_2,\; H_3,\; H_1 + H_2 + H_3 - E_4$ & $H_1,\; H_2,\; H_3,\; E_4$ \\ \hline
        $4.6$  & $H_1,\; H_2,\; H_3,\; H_1 + H_2 + H_3 - E_4$ & $H_1,\; H_2,\; H_3,\; E_4$ \\ \hline
        $4.7$  & $H_1,\; H_2,\; H_2 - E_2,\; H_1 - E_1$ & $H_1,\; H_2,\; E_1,\; E_2$ \\ \hline
        $4.8$  & $H_1,\; H_2,\; H_3,\; H_1 + H_2 + H_3 - E_1$ & $H_1,\; H_2,\; H_3,\; E_1$ \\ \hline
        $4.12$ & $H,\; H - E_1 - E_2 - E_3,\; H - E_2,\; H - E_1$ & $H,\; E_1,\; E_2,\; E_3$ \\ \hline
    \end{tabular}
\end{table}

where $H,\; H_{i}$ denote the pullback to $X$ of a general hyperplane in the base threefold and $E,\; E_j$ denote the exceptional divisors of blow-ups. In each case, we observe that these generators of $\nef(X) \cap \pic(X)$, in turn, are (extremal) generators of the entire Nef cone $\nef(X)$ and they form a $\zz$-basis of the Picard group. We know that the $G$-action on the generators of $\nef(X)$ is by permutations, it induces a permutation action on $\pic(X)$. This action is explicitly described in \S 4 in each case. 

Consider now the following families
    \begin{center}
        \textnumero 3.9, \textnumero 3.19,  \textnumero 4.2, \textnumero 4.4, \textnumero 5.1, \textnumero 5.2. 
    \end{center}
If $X$ is a smooth member of any of these families, then $\rho <n$. Assume $X$ is a Fano threefold in families \textnumero 3.19, \textnumero 4.4, \textnumero 5.1, \textnumero 5.2. We can again find the generators of $\nef(X) \cap \pic(X)$ from Matsuki's tables. These also happen to be extremal and they span $\nef(X)$. We list the corresponding $\zz$-basis of $\pic(X)$ for each family in the second column of the table below. In the third column, we list generators of $G$, acting on $ \pic(X)$ as a permutation $\zz[G]$-module. These generators describe the action of $G$ on $\pic(X)$ by permuting the divisor classes corresponding to the basis elements in column 2 (e.g., $(23)$ exchanges the second and third).

\begin{table}[ht]
    \centering
    \renewcommand{\arraystretch}{1.3} 
    \setlength{\tabcolsep}{10pt} 
    \begin{tabular}{|c|c|c|}
        \hline
        MM \textnumero  & $\pic(X)$ basis & Generators of $G$ \\
        \hline
        $3.9$  & $2H-E_4,\; E_1-E_4,\; H$   & $(12)$\\   \hline
        $3.19$  & $H,\; E_1,\; E_2$ & $(23)$\\ \hline
        $4.2$ & $H_1,\;H_2,\; \widetilde{S_+},\; E_5$   & $(12),(34)$ \\    \hline
        $4.4$  & $H,\; E_1,\; E_2,\; E_5$ & $(23)$\\
             \hline
        $5.1$  & $H,\; E_1,\; E_2, E_3,\; E_7$ & $(234),\;(23) $\\
        \hline
        $5.2$  & $H,\; E_1,\; E_2, E_3,\; E_5$ & $(23)$ \\ 
              \hline
    \end{tabular}
\end{table}

The detailed description of divisors that appear in each case can be found in \cite{Mat95,Mat23}. The induced action of $G$ is by permutations follows as a consequence (see also \S 4). 

If $X$ is Fano threefold in \textnumero 4.10 or \textnumero 5.3, then it is toric. In Lemmas \ref{autp:3-31,4-10,4-12,5-2} and \ref{autp:5-3}, we show the existence of a basis for $\pic(X)$ which is permuted under the group $\autp(X)$. 
This completes the proof.
\end{proof}

From Lemmas \ref{std-permutations-module-fact} and \ref{main-H1-result}, we now have 
\begin{coro}
    If $X$ is a Fano threefold with Picard rank $\rho\leq 5$ and $G\subseteq \aut(X)$ is a finite group, then $H^1(G,\pic(X))=0$. 
\end{coro}

\newpage
\section{Table}  \label{table}
We summarize the results from previous sections in the table below. In column 1, we write the Mori-Mukai name for the deformation family of Fano threefolds. Column 2 lists $\autp(X)$, the largest group that can be realized as the image of the action of $\aut(X)$ on $\pic(X)$. Column 3 lists Matsuki's Weyl group from \cite{Mat23}. In column 4, we list the result where we show if $\autp(X)=\wg_X$ can be realized for that family.

\small{\begin{longtable}{|p{2.5cm}|p{2.0cm}|p{2.7cm}|p{2.5cm}|}
 \caption{$\autp(X)$ for smooth Fano threefolds}\\
    \hline
MM \textnumero  &  $\autp(X)$  & $\wg_X$ & Refer \\
\hline
1.1, \dots , 1.17   &  $0$   &  $0$  &   \S 3 \\
\hline 
2.1   &  $0$  &  $0$  & Corollary \ref{im=0forsure}\\
\hline
2.2  &   $0$   & $\zz/2\zz$   &  Remark \ref{autp:2.2exception} \\
\hline
2.3, 2.4, 2.5 & $0$  &  $0$  & Corollary \ref{im=0forsure}\\
\hline
2.6    &  $\zz/2\zz$ & $\zz/2\zz$  & Lemma \ref{autP:2-6a)}, \ref{autP:2-6b)} \\
\hline
2.7, \dots 2.11    & $0$ & $0$  & Corollary \ref{im=0forsure}  \\
\hline
2.12   & $\zz/2\zz$ & $\zz/2\zz$  &  Lemma \ref{autP:2-12}   \\
\hline
2.13, \dots , 2.20    & $0$ & $0$  & Corollary \ref{im=0forsure}    \\
\hline
2.21   & $\zz/2\zz$ & $\zz/2\zz$  & Lemma \ref{autP:2-21}  \\   
\hline
2.22 \dots , 2.31    & $0$ & $0$  & Corollary \ref{im=0forsure}  \\
\hline
2.32    & $\zz/2\zz$ & $\zz/2\zz$  & Lemma \ref{autP:2-32}   \\
\hline
2.33 \dots , 2.36    & $0$ & $0$  & Corollary \ref{im=0forsure} \\
\hline
3.1   & $S_3$ & $S_3$  & Lemma \ref{autP:3-1} \\
\hline
3.2   & $0$ & $0$  & Corollary \ref{im=0forsure}   \\
\hline
3.3  & $\zz/2\zz$ & $\zz/2\zz$ &  Lemma \ref{autP:3-3}   \\
\hline
3.4, \dots , 3.6    & $0$ & $0$  & Corollary \ref{im=0forsure}    \\
\hline
3.7  & $\zz/2\zz$ & $\zz/2\zz$  & Lemma \ref{autP:3.7}    \\
\hline
3.8  & $0$ & $0$  & Corollary \ref{im=0forsure}   \\
\hline
3.9  & $\zz/2\zz$ & $\zz/2\zz$ & Lemma \ref{autP:3.9}\\
\hline
3.10  & $\zz/2\zz$ & $\zz/2\zz$  & Lemma \ref{autP:3-10}   \\
\hline
3.11, 3.12  & $0$ & $0$  &  Corollary \ref{im=0forsure}  \\
\hline
3.13  & $S_3$ & $S_3$  & Lemma \ref{autP:3.13}   \\
\hline
3.14, 3.15, 3.16  & $0$ & $0$  & Corollary \ref{im=0forsure}   \\
\hline
3.17  & $\zz/2\zz$ & $\zz/2\zz$  & Lemma \ref{autP:3-17}    \\
\hline
3.18  & $0$ & $0$  & Corollary \ref{im=0forsure}   \\
\hline
3.19  & $\zz/2\zz$ & $\zz/2\zz$  & Lemma \ref{autP:3-19}   \\
\hline
3.20  & $\zz/2\zz$ & $\zz/2\zz$  & Lemma \ref{autP:3-20}    \\
\hline
3.21, \dots , 3.24  & $0$ & $0$  & Corollary \ref{im=0forsure}    \\
\hline
3.25  & $\zz/2\zz$ & $\zz/2\zz$  & Lemma \ref{autP:3-25}  \\
\hline
3.26  & $0$ & $0$  & Corollary \ref{im=0forsure}    \\
\hline
3.27  & $S_3$ & $S_3$  &  Example \ref{autP:3.27}   \\
\hline
3.28, 3.29, 3.30  & $0$ & $0$  & Corollary \ref{im=0forsure}\\
\hline
3.31  & $\zz/2\zz$ & $\zz/2\zz$  & Lemma \ref{autp:3-31,4-10,4-12,5-2}   \\
\hline
4.1  & $S_4$ & $S_4$  & Lemma \ref{autP:4-1}   \\
\hline
4.2  & $(\zz/2\zz)^2$ & $(\zz/2\zz)^2$  & Lemma \ref{autP:4.2}\\
\hline
4.3  & $\zz/2\zz$ & $\zz/2\zz$  &  Lemma \ref{autP:4.3}   \\
\hline
4.4  & $\zz/2\zz$ & $\zz/2\zz$  & Lemma \ref{autP:4.4}   \\
\hline
4.5  & $0$ & $0$  & Corollary \ref{im=0forsure}   \\
\hline
4.6  & $S_3$ & $S_3$  & Lemma \ref{autP:4.6}    \\
\hline
4.7  & $\zz/2\zz$ & $\zz/2\zz$  & Lemma \ref{autP:4.7}    \\
\hline
4.8  & $\zz/2\zz$ & $\zz/2\zz$  & Lemma \ref{autP:4.8}    \\
\hline
4.9  & $0$ & $0$  & Corollary \ref{im=0forsure}    \\
\hline
4.10  & $\zz/2\zz$ & $\zz/2\zz$  &  Lemma \ref{autp:3-31,4-10,4-12,5-2}  \\
\hline
4.11  & $0$ & $0$  & Corollary \ref{im=0forsure} \\
\hline
4.12  & $\zz/2\zz$ & $\zz/2\zz$  & Lemma \ref{autp:3-31,4-10,4-12,5-2}   \\
\hline
4.13  & $\zz/2\zz$ & $\zz/2\zz$  & Lemma \ref{autP:4.13}    \\
\hline
5.1  & $S_3$ & $S_3$  & Lemma \ref{autP:5.1}   \\
\hline
5.2  & $\zz/2\zz$ & $\zz/2\zz$  & Lemma \ref{autp:3-31,4-10,4-12,5-2}   \\
\hline
5.3  & $\zz/2\zz \times S_3$  & $\zz/2\zz \times S_3$  & Lemma \ref{autp:5-3}    \\
\hline
6.1 & $S_5$ &  $S_5$ & \S \ref{autP:P1timesdP}, \cite{dolgaisko_planecremonagroups}   \\
\hline 
7.1 & - &  $W(D_5) $ & \S \ref{autP:P1timesdP}, \cite{dolgaisko_planecremonagroups}  \\
\hline
8.1  & - & $W(E_6)$  &  \S \ref{autP:P1timesdP}, \cite{dolgaisko_planecremonagroups}\\
\hline
9.1  & -  & $W(E_7)$  &  \S \ref{autP:P1timesdP}, \cite{dolgaisko_planecremonagroups}\\
\hline 
 10.1  & - & $W(E_8)$ & \S \ref{autP:P1timesdP}, \cite{dolgaisko_planecremonagroups} \\
\hline
\end{longtable}}

\bibliographystyle{alphaurl} 
\bibliography{bibis} 

\end{document}